\documentclass[12pt]{article}
\usepackage{amsmath,amssymb}
\usepackage{amsthm}
\usepackage{epsfig}

\newtheorem{theorem}{Theorem}
\newtheorem{proposition}{Proposition}
\newtheorem{corollary}{Corollary}

\def\ds{\displaystyle}
\def\n{\noindent}

\smallskip
\def\bsq{\blacksquare}
\def\eproof{$\hfill\bsq$\par}

\begin{document}

\title{Things to do with a broken stick
\footnote{or ``Geometric probabilities for triangle constructions''}}

\author{
 Eugen J.~Iona\c{s}cu and Gabriel Pr\u{a}jitur\u{a}
}
\date{April $20^{th}$, 2013}

\maketitle

\section{INTRODUCTION}

The following problem, sometimes called the
spaghetti problem, goes back to at least 1854, being included in
\cite{me} on page 49. A good historical account can be found in a
recent paper of Goodman (\cite{gg}). For a long time it has
captured the attention of various mathematicians and educators and
it seems to have stirred quite an interest in more recent years
(see \cite{dg}, \cite{bb}, \cite{gg}, \cite{guy}, \cite{di},
\cite{j}, and \cite{lw}). The following formulation is probably closer than others to
Martin Gardner's prefernece (\cite{mg}, \cite{di}):

\begin{quote}
{\bf The Broken Stick Problem:} {\it A spaghetti stick, dropped on
the floor, breaks at random into three pieces. What is the
probability that the three parts obtained are the sides of a
triangle?}
\end{quote}

The formulation of this problem in \cite{me} is a little different
but illuminating: {\it ``A rod is marked at random at two points,
and divided into three parts at those points; shew that the
probability of its being possible to form a triangle with the
pieces is $\frac{1}{4}$."}  In this paper, we will consider other problems that start
with breaking a stick into three pieces, and using the resulting lengths to construct a
triangle. There are many different ways in which a triangle can be
constructed from three segments. Such a triangle can be defined more or less in an unique way, i.e. it can be identified up to congruency or similarity. In the original problem, the segments become the
sides, but they might also become the medians of the triangle, or the angle bisectors,
or even other parameters of the triangle such as angles. In each case we can ask for the
probability that a triangle can actually be constructed from these measurements, and the
probability that the triangle is acute. For example, in the original problem stated above, the triangle is acute with probability
$\ln (8/e^2)$. In some cases the probabilities considered
are difficult to compute, if not impossible, and in those
situations we find only approximations for them or their
experimental frequencies. To give the reader more inside and to challenge
him/her at the same time, we next include one such problem
discussed briefly in Section~\ref{AngleBisectors}:

\begin{quote}
{\it A stick is broken into three pieces at random. Show that the
probability that the three parts obtained are the angle bisectors
of a triangle is equal to one. }
\end{quote}

The triangle mentioned above is uniquely determined and the
probability that it is acute is about $0.1195$ (found
experimentally). The exact value of this probability is yet
unknown to us. One other intriguing fact that we discuss in more detail in Section~\ref{mediansCase}, is that
the probability for the existence of a triangle  whose medians are the three parts of the stick is still $\frac{1}{4}$.
Moreover, the probability that this triangle exits and it is an acute triangle equals
$$\frac{1}{3}-\frac{5}{9}\ln\left(\frac{8}{5}\right)\approx 0.0722202059.$$

For a summary of the geometric  probabilities calculated here or left as exercises
for the reader in Section~\ref{challengeproblems}, one may go directly to the table at the end of the paper.

In what follows we are going to adopt the standard notations for
the elements in an arbitrary triangle $ABC$: $a$, $b$, and $c$ for
the sides, $A$, $B$, and $C$ for its angles (measured in radians),
$m_a$, $m_b$, and $m_c$ for the lengths of its medians, $h_a$,
$h_b$, and $h_c$ for the lengths of its altitudes, $w_a$, $w_b$,
and $w_c$ for the lengths of its angle bisectors, $K$ for its
area, $R$ and $r$ for the radii of the circumcircle and the
incircle, $O$ the center of the circumcircle and $I$ for the
center of the incircle.

\section{About our probabilistic model}\label{firstsection}

As with most geometric probabilities, it is important to be very
specific about how the random concept is defined--- in our case, as to how the two breaking points are chosen.
It is natural to consider that these points are simultaneously chosen at random with uniform distribution.  How do we accomplish this,
is on one hand of theoretical importance and on another, useful for experimental simulations that should match our exact calculations. A simple way to do this and an equivalent one is to choose
a point $(x,y)$ uniformly from the square $[0,L]^2\subset \mathbb R^2$, where $L$ is the stick length, and then break the stick at $x$ and $y$. Since the endpoints are perfectly symmetric we cannot distinguish between $(x,y)$ and $(y,x)$.

However, our approach here is different but the idea is nevertheless a classical one. Surprisingly enough (see \cite{gg} and \cite{p}),
$\rm{Poincar\acute{e}}$ was the first to use this idea and showed that it indeed models the stick problem in the sense stated above.

Without loss of generality, we will assume the stick has a length $L:=\sqrt{3}$. The procedure
of obtaining the three broken parts, of lengths $\alpha$, $\beta$,
and $\gamma$, and with these parts positioned in order from left
to right, say, on a horizontal stick, is the following.

Let $ABC$ denote an equilateral triangle with side lengths equal
to $2$ (Figure~\ref{Figure1}(a)), having coordinates $A(1,0)$, $B(-1,0)$ and $C(0,\sqrt{3})$. We choose a point $Q$ uniformly distributed inside of the triangle $ABC$.
Then the three parts of the stick are just the distances,
$\alpha=QM$, $\beta=QN$, and $\gamma=QP$ from $Q$ to the sides of
the equilateral triangle $ABC$. Viviani's Theorem (see \cite{kkaw} and \cite{hsam}) tells us that
indeed $\alpha+\beta+\gamma=\sqrt{3}$. Then, the way we are going to
calculate the probability of an event $E$ is first to determine
the region $\mathcal R$ inside of the equilateral triangle $ABC$
that characterizes it, and then put $P(E):=\frac{Area({\mathcal
R})}{\sqrt{3}}$ since the area of the triangle $ABC$ is
$\sqrt{3}$. $\rm{Poincar\acute{e}}$ (\cite{p}) has shown that this is a perfect model for the stick problem, bringing beauty, symmetry
and easiness in calculations to all of our variations considered here.

\begin{figure}
\[
\underset{(a)}{\epsfig{file=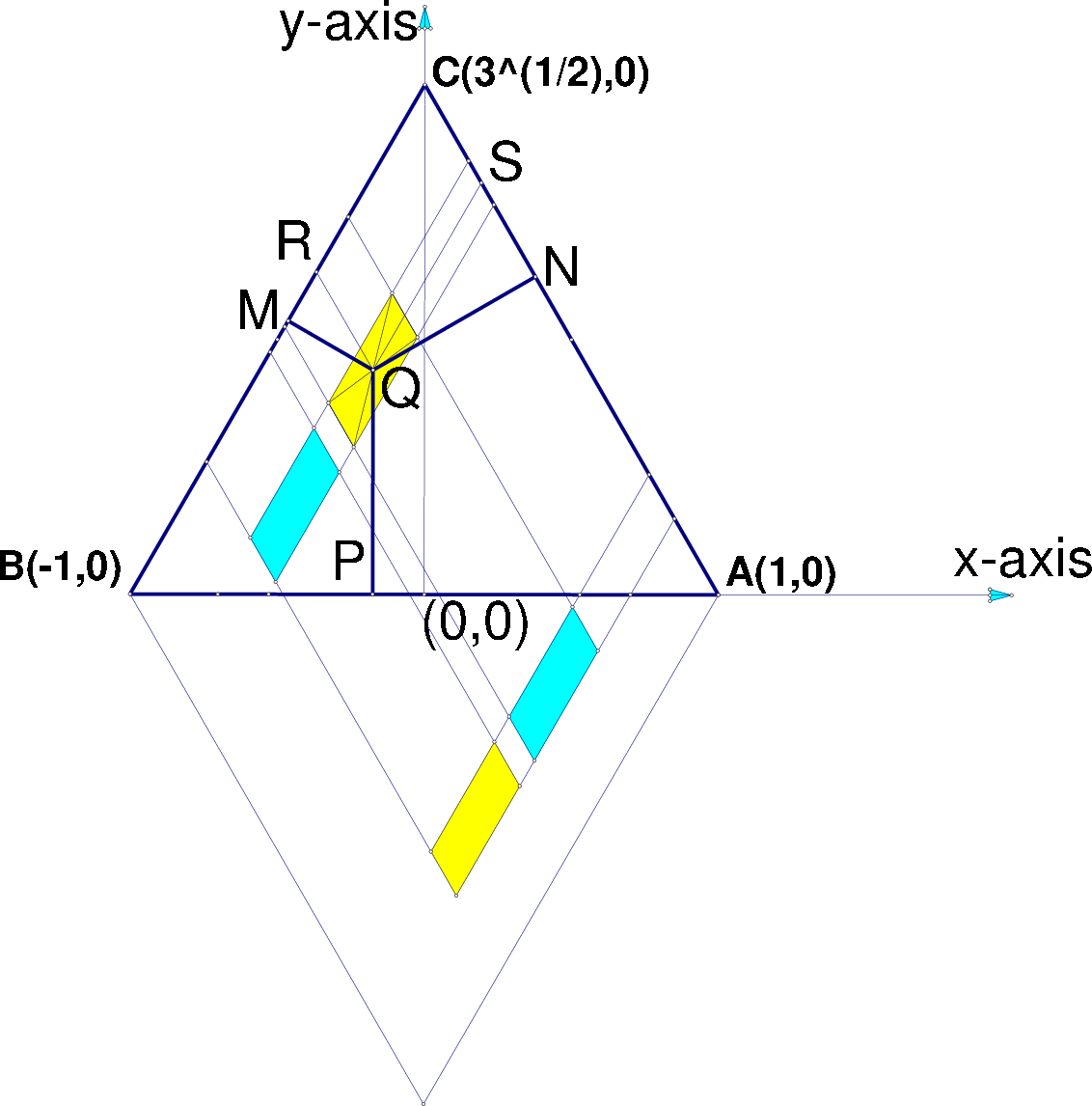,height=2.2in,width=2.6in}}
\ \
\underset{(b)}{\epsfig{file=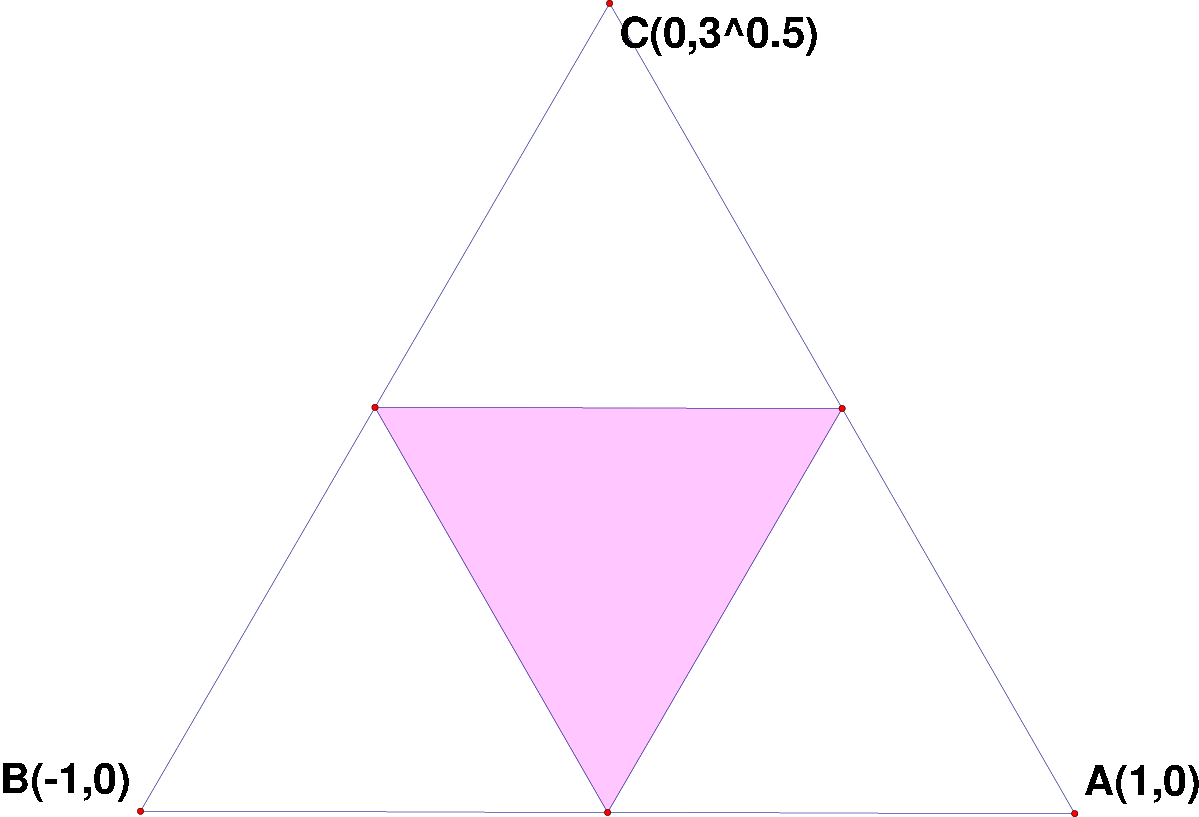,height=1.57in,width=2.46in}}
\]
\caption{Distances $\alpha=QP,\ \beta=QM,\ \text{and}\ \gamma=QN$,
add up to $\sqrt{3}$.} \label{Figure1}
\end{figure}

From here on we will refer to this model whenever we have a
probability question which involves three positive quantities
which add up to a constant value.   Let us make a few more observations about Figure~\ref{Figure1}(a). Let $(x,y)$ be the generic Cartesian coordinates of the
point $Q$ relative to a system of orthogonal axes. We observe that $x\in [-1,1]$ and
$y\in [0,\sqrt{3}]$. In general, the distance $d$ from a point with coordinates $(x_0,y_0)$ to
the line $mX+mY+p=0$ is given by the formula
$d=\frac{|mx_0+ny_0+p|}{\sqrt{m^2+n^2}}.$ Since $\overline{BC}$
has equation $\sqrt{3}(X+1)-Y=0$ and $\overline{AC}$ has equation
$\sqrt{3}(1-X)-Y=0$, we obtain

\begin{equation}\label{definitionofalphabetagamma}
\alpha=y, \  \ \beta=\frac{\sqrt{3}(1+x)-y}{2}\ \  \text{and} \
 \ \gamma=\frac{\sqrt{3}(1-x)-y}{2}.
\end{equation}

We will do all of the computations needed in terms of $x$ and $y$,
taking advantage, in most of the situations, of the symmetries of
the region involved such as $120^{\circ}$ rotational invariance.

We divided the rest of this article into three sections. The first is designed
to give situations where exact calculations can be done, the
second contains various experimental cases, and finally we end with a section in which we summarize
the probabilities included here, some other results whose proofs can be found elsewhere,
some open questions and further lines of investigation.

\section{EXACT CALCULATIONS}\label{exactcalculations}

To exemplify our model,  let us look at the original stick problem---the region
that describes the event that a triangle with sides $\alpha$,
$\beta$ and $\gamma$ exists is given by the triangle
inequality, which in turn can be written as $\max(\alpha,\beta,\gamma)<
(\alpha+\beta+\gamma)/2=\sqrt{3}/2$. This gives
the interior of the triangle determined by the midpoints of the
sides $\overline{AB}$, $\overline{CB}$ and $\overline{CA}$ as depicted in
Figure~\ref{Figure1}(b). Hence, the probability of having a
triangle with $\alpha$, $\beta$ and $\gamma$ as its side lengths
is equal to $1/4$.

Let us observe that, it is still the same probability (and idea of
proof) for the existence of an acute triangle with angles (in
radians) of $\frac{\alpha\pi}{\sqrt{3}}$,
$\frac{\beta\pi}{\sqrt{3}}$, and $\frac{\gamma\pi}{\sqrt{3}}$. As a result we find
that {\it ``There are three times as many obtuse-angled triangles as there
are acute-angled ones"} as Richard Guy found in \cite{guy}. In what will follow we will
look at this ratio, between the probability of obtaining an obtuse triangle versus an isosceles one,
from different constructions. We will see that this ratio may take unexpected values (far away from 3) depending upon the construction used.

 Other authors, see \cite{guy},
\cite{i}, and \cite{lw}, have looked into similar questions, but our
technique is nevertheless the first that goes through a
significant number of such problems and provides a common approach
for their solutions.  In \cite{lw}, for instance, it is shown that the probability that an acute triangle of sides $\alpha$, $\beta$ and $\gamma$ exists is
$2(-\ln(1/2)-1+\text{arccosh}\left(3\frac{\sqrt{2}}{4}\right))$.
We do the calculations for this problem in the next section for
completeness, our method being quite shorter than the one used in
\cite{lw} and, also because our answer, although the same, turns
out to be $3\ln 2-2 \approx 0.079441$.

There are a few natural questions along the
lines specified in the Introduction  in which the probabilities
involved turn out to have interesting expressions in terms of
known constants and these are going to be included in the next four subsections.

\vspace{0.2in}

\subsection{The Sides}\label{thesides} We have already analyzed the
classical problem and the reader can find various approaches to it
in \cite{gr} and \cite{me}.

Let us continue with our initial classical problem and see what
happens in the special case when the triangle constructed with
$\alpha$, $\beta$, and $\gamma$ is acute.  A similar probability is
studied in \cite{es} in  Euclidean geometry and in \cite{i} in
hyperbolic geometry.

\begin{figure}[h]
 $$\underset{\ \text {Acute triangle with $\alpha$, $\beta$ and
$\gamma$ as sides}
}{\epsfig{file=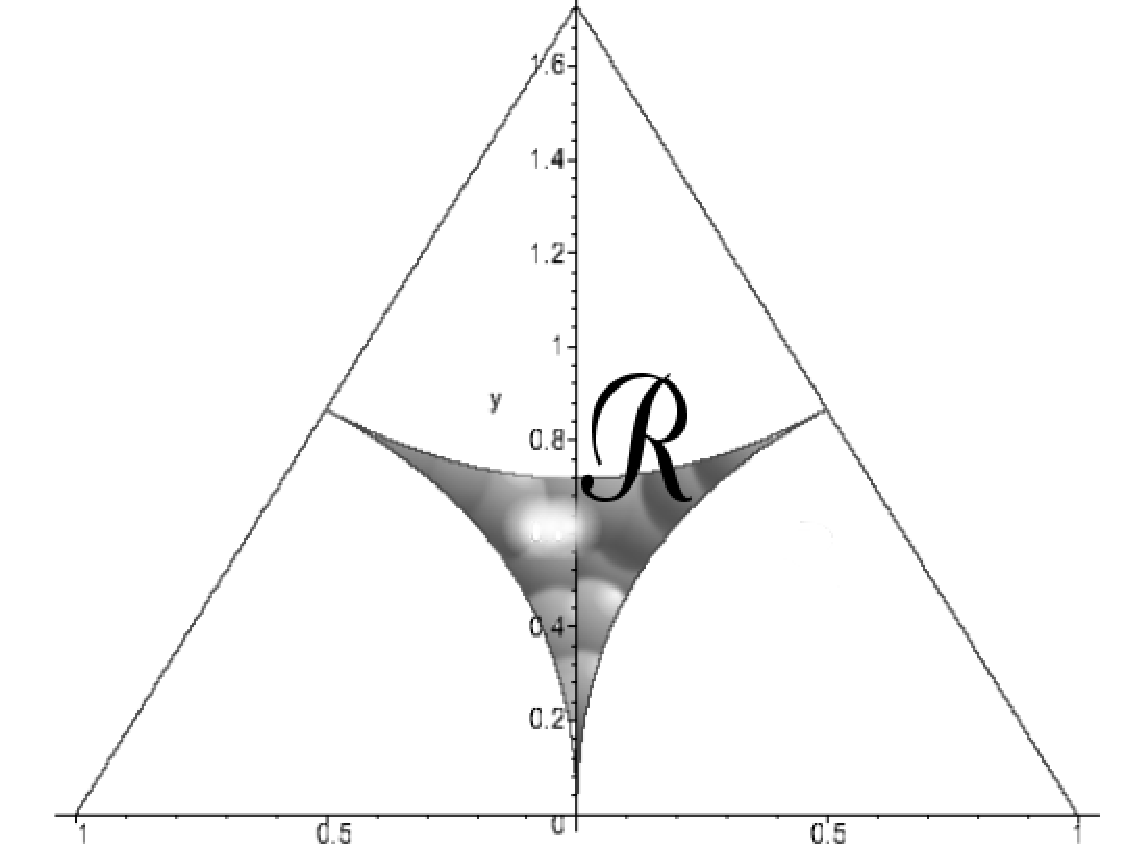,height=1.6in,width=2in}}$$\
\ \  \caption{
$\frac{1}{\sqrt{3}}Area(\mathcal R)=3\ln 2-2$.} \label{Figure2}
\end{figure}

\begin{theorem}\label{sidesacute} The probability that the three parts of the broken
stick form an acute triangle is equal to $\ln \left(8/e^2\right)$.
\end{theorem}

\begin{proof} We need to find the area of the region
$\mathcal R$ (see Figure~\ref{Figure2}(b)), described by
$\alpha^2+\beta^2-\gamma^2>0$, $-\alpha^2+\beta^2+\gamma^2>0$, and
$\alpha^2-\beta^2+\gamma^2>0$. This region is bounded by three
hyperbolae  which pass through the midpoints of the sides and
intersect only at these points as shown in Figure~\ref{Figure2}(b). The inequality $\alpha^2-\beta^2+\gamma^2>0$ becomes
$y^2-3x+\sqrt{3}xy>0$ if we use the substitutions from
(\ref{definitionofalphabetagamma}), or
$x<\frac{y^2}{\sqrt{3}(\sqrt{3}-y)}$. So, the probability we are
interested in is

$$P=\frac{1}{\sqrt{3}}\left(\frac{\sqrt{3}}{4}-3\int_0^{\sqrt{3}/2}\frac{y}{\sqrt{3}}-
\frac{y^2}{\sqrt{3}(\sqrt{3}-y)}dy\right)=3\ln 2-2. $$
 \end{proof}
This probability was also obtained implicitly by
Richard Guy in \cite{guy}, where he  looked at some other ways
of constructing a triangle besides the broken stick approach.
Guy gives the value of $\frac{\frac{1}{4}-P}{\frac{1}{4}}=9-12\ln
2$ representing the conditional probability that an obtuse
triangle is obtained, knowing that the three parts of the stick
already form a triangle. In this
situation we obtain
$$\frac{P(obtuse)}{P(acute)}=\frac{9-12\ln 2}{12\ln 2-8}\approx
2.146968.$$

\subsection{MEDIANS}\label{mediansCase}  There is a well known theorem in geometry stating that if one constructs a triangle using the medians of a given
triangle and then does that again, i.e. constructs a triangle with
the new medians, the result is a triangle similar to original
triangle and the similarity ratio is $\frac{3}{4}$
(Figure~\ref{Figure3} (b)). This explains at least the first part
of the next result.

\begin{figure}
$\underset{ \text {(a)\ Acute triangles with $\alpha$, $\beta$ and
$\gamma$ as medians}
}{\epsfig{file=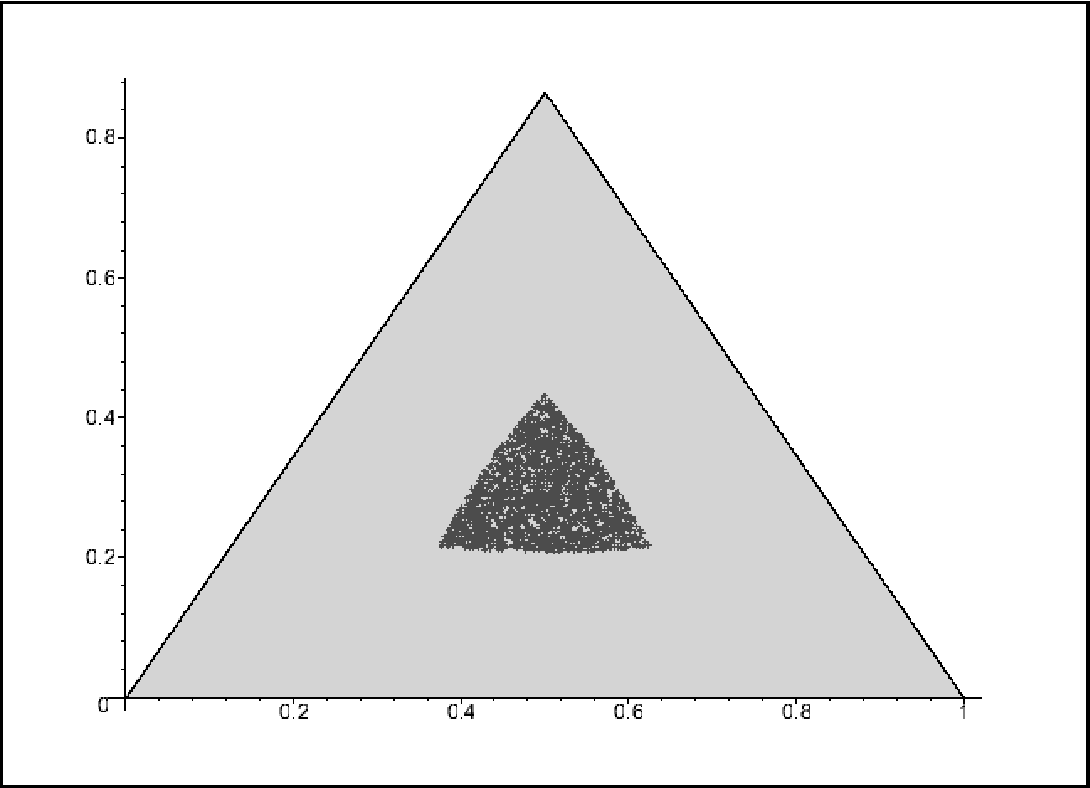,height=2in,width=2.3in}}$\ \ \
$\underset{(b)\ \text{Two iterations bring back the shape}
}{\epsfig{file=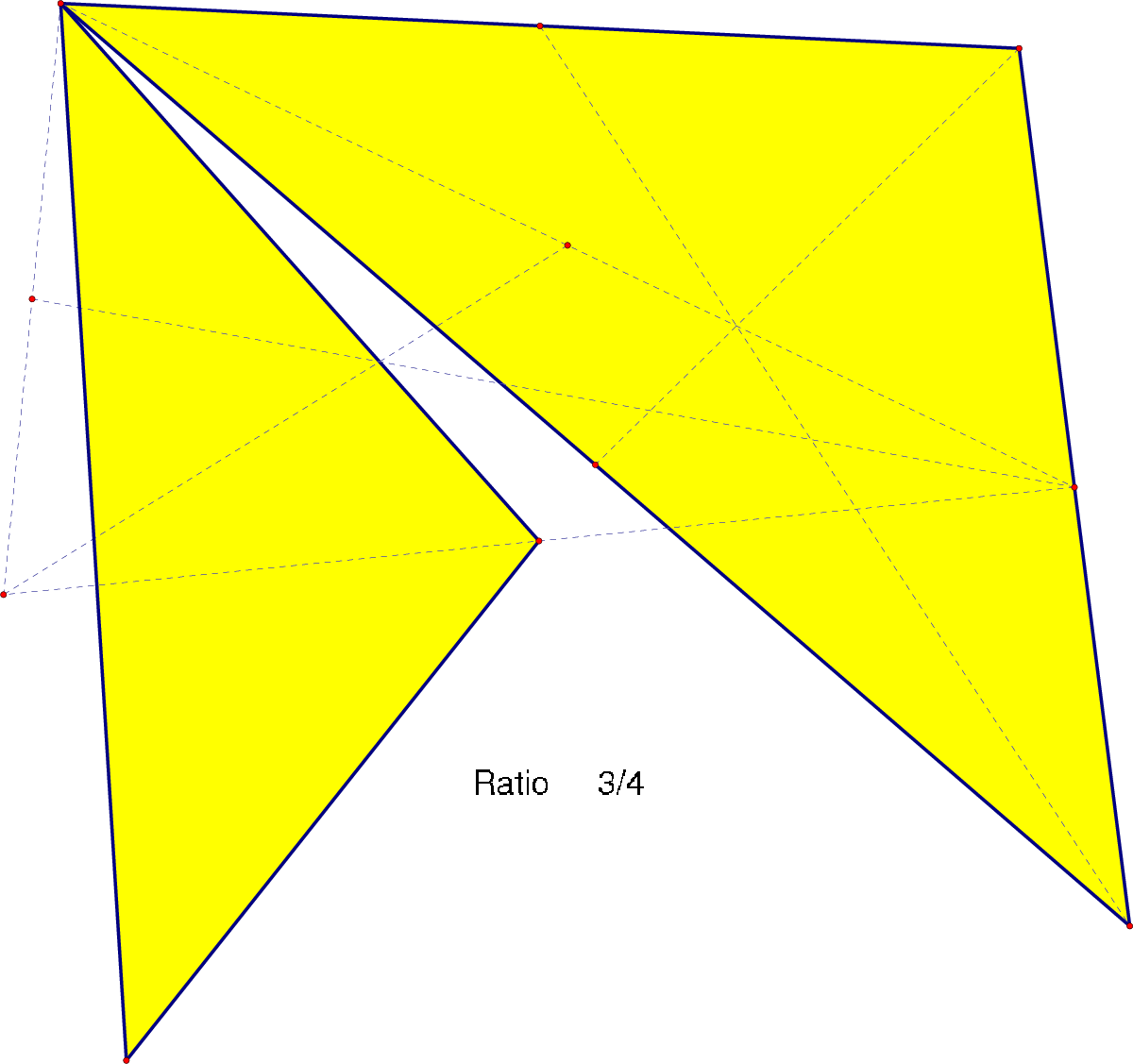,height=2in,width=2.1in}}$
\caption{Medians} \label{Figure3}
\end{figure}

\begin{theorem} Given three positive quantities $u$, $v$ and $w$,
there exists a triangle whose medians are precisely $u$, $v$ and
$w$ if and only if $u+v+w>2\max(u,v,w)$. If the triangle exists,
it is unique.  Moreover, the triangle is acute if and only if
$u^2+v^2+w^2<6 \min(u^2,v^2,w^2)$.
\end{theorem}
\begin{proof} The formula which gives the medians in terms of the
sides of the triangle $ABC$ is $m_a^2=\frac{2(b^2+c^2)-a^2}{4}$.
This implies that
$a^2=\frac{4}{9}\left(2(m_b^2+m_c^2)-m_a^2\right)$ and the other
analogous relations for $b^2$ and $c^2$. If $m_a=u$, $m_b=v$ and
$m_c=w$, the inequality $a+b>c$ is equivalent to

\begin{equation}\label{mediansid}
2\sqrt{2v^2+2w^2-u^2}\sqrt{2u^2+2w^2-v^2}>(u^2+v^2+w^2)-6w^2.
\end{equation}

We note that if $u^2+v^2+w^2<6\ min(u^2,v^2,w^2)$ then the above
inequality is true. So, for the second part of the statement, the
hypothesis implies by itself the existence of the triangle with
medians $u$, $v$ and $w$. We may suppose then that $u^2+v^2+w^2\ge
6\ min(u^2,v^2,w^2)$. Assuming, without loss of generality, that
$w\le v\le u$, we have $u^2+v^2+w^2\ge 6w^2$. This means we can
continue with (\ref{mediansid}) by squaring both sides and get

 $$(u+v+w)(u+v-w)(u-v+w)(-u+v+w)>0.$$

Certainly, under the hypothesis that $w\le v\le u$, the above
translates into $v+w>u$ or $u+v+w>2\max(u,v,w)$. Now, if
$u^2+v^2+w^2<6 \min(u^2,v^2,w^2)$ we can show that this also
implies  $v+w>u$. Indeed, if $u^2+v^2<5w^2$ then $u^2<5w^2-v^2\le
(v+w)^2$ and so $u<v+w$. Hence, in any case we must have
$u+v+w>2\max(u,v,w)$. This proves the ``necessary" part of the
first statement in our theorem.

For the converse let us observe that the formulae for $a$, $b$ and
$c$ in terms of $u$, $v$ and $w$, i.e.
$a^2=\frac{4}{9}\left(2(v^2+w^2)-u^2\right)$, etc., make sense
because, for instance, $2(v^2+w^2)\ge (v+w)^2>u^2$. The triangle
inequality, $a+b+c>2\max(a,b,c)$, follows from the work we did
earlier.

For the second part of the statement, one has to observe that
$a^2+b^2-c^2>0$ is equivalent to $u^2+v^2<5w^2$.
\end{proof}

\begin{corollary}\label{medians} The probability that the three parts of the broken
stick are the medians of a triangle  is $\frac{1}{4}$. Moreover, the
probability that this triangle is acute equals
$$\frac{1}{3}-\frac{5}{9}\ln\left(\frac{8}{5}\right)\approx 0.0722202059.$$
\end{corollary}

\begin{proof} The first part of this corollary follows from what we did earlier. The region that
defines the acute triangles with medians $\alpha$, $\beta$ and
$\gamma$ is depicted in Figure~3(a). This region is bounded by
the curves $\alpha ^2+\beta ^2+\gamma^2=6\alpha^2$, $\alpha
^2+\beta ^2+\gamma^2=6\beta^2$ and $\alpha ^2+\beta
^2+\gamma^2=6\gamma^2$, each pair intersecting at points such as
$\alpha=\beta=\frac{\sqrt{3}}{4}$, $\gamma=\frac{\sqrt{3}}{2}$ and
the other  corresponding cyclic permutations. Using
(\ref{definitionofalphabetagamma}), the curve
$\beta^2+\gamma^2=5\alpha^2$ in terms of $x$ and $y$ is the
hyperbola

$$y=\frac{1}{3\sqrt{3}}(\sqrt{9x^2+10}-1)$$

\n which gives the probability

$$P=\frac{1}{\sqrt{3}}\left(\frac{\sqrt{3}}{16}+3\int_{-1/4}^{1/4}\frac{\sqrt{3}}{4}-\frac{\sqrt{9x^2+10}-1}{3\sqrt{3}}dx\right).$$
One can use the formula

\begin{equation}\label{indefiniteint}
\int\sqrt{x^2+k}dx=\frac{1}{2}x\sqrt{x^2+k}+\frac{k}{2}\ln(x+\sqrt{x^2+k})+C, \ k\in \mathbb R,
\end{equation}

\n to compute this last integral and simplify it to the expression
in the statement of the corollary. \end{proof}

In this case, the ratio between obtuse versus acute  is equal to

$$\frac{P(obtuse)}{P(acute)}=\frac{3-60\ln 2+20\ln 5}{60\ln 2-20\ln 5-12} \approx
2.461635121 $$ \vspace{0.2in}

\vspace{0.1in}

\subsection{The altitudes}\label{altitudesCase} There are
fairly complicated formulas that give the sides $a$, $b$ and $c$
of a triangle in terms of its altitudes $h_a$, $h_b$ and $h_c$.
However, the existence of $a$, $b$ and $c$  is given by a very
basic condition which allows a closed form for the desired
probability.

\begin{theorem}\label{altitudestheorem}  A stick is broken into three pieces, $\alpha$, $\beta$ and $\gamma$, at random (as described earlier).
\par (i) The probability that $\alpha$, $\beta$ and $\gamma$ are the heights of
a triangle is equal to
$$\frac{4}{25}\left(3\sqrt{5}\ln \frac{3+\sqrt{5}}{2}-5\right). $$
\par (ii) The probability that $\alpha$, $\beta$ and $\gamma$ are the heights of an acute triangle is equal
to

$$1-2\sqrt{3}\int_0^{\frac{2\sqrt{6}-\sqrt{3}}{7}}\left(15t^2-6\sqrt{3}t+9-12t(2t^2-2\sqrt{3}t+3)^{\frac{1}{2}}
\right)^{\frac{1}{2}}\ dt\approx 0.07744388\ . $$
\end{theorem}

In Figure~\ref{figure4}(a) we have depicted the region corresponding to this event.

\begin{proof}   (i) In Figure~\ref{figure4}(a) we have
depicted the region corresponding to this event. The lengths
$\alpha$, $\beta$ and $\gamma$ are the heights of a triangle if
and only if $a=\frac{2S}{\alpha}$, $b=\frac{2S}{\beta}$, and
$c=\frac{2S}{\gamma}$ satisfy the triangle inequality. This is
equivalent to

$$\frac{1}{\alpha}+\frac{1}{\beta}+\frac{1}{\gamma}> 2 \max(\frac{1}{\alpha},
\frac{1}{\beta},\frac{1}{\gamma}).$$

We are going to evaluate the probability of the complementary
event:

$$\underset{(\star)}{\underbrace{\frac{1}{\alpha}+\frac{1}{\beta}\le \frac{1}{\gamma}}},\ 
\underset{(\star\star)}{\underbrace{\frac{1}{\beta}+\frac{1}{\gamma}\le
\frac{1}{\alpha}}},\ \text{or} \ \
\underset{(\star\star\star)}{\underbrace{\frac{1}{\alpha}+\frac{1}{\gamma}\le
\frac{1}{\beta}}}.$$

Because of the symmetry of the problem, we will just work with
($\star\star$) using the formulas in
(\ref{definitionofalphabetagamma}):

$$\frac{1}{\beta}+\frac{1}{\gamma}\le \frac{1}{\alpha}\ \Leftrightarrow \
4(\sqrt{3}-y)y\le (\sqrt{3}-y)^2-3x^2\ \Leftrightarrow 3x^2\le
3-6\sqrt{3}y+5y^2.$$

Equivalently, $3x^2\le (\sqrt{3}-y)(\sqrt{3}-5y)$ implies in
particular that $0<y\le \frac{\sqrt{3}}{5}$. Then, we can solve
for $y$ to obtain $$0< y\le
\frac{3\sqrt{3}}{5}-\frac{\sqrt{15x^2+12}}{5}.$$ The graph of the
equation $y=(3\sqrt{3}-\sqrt{15x^2+12})/5$, shown in
Figure~\ref{figure4}(a) as the south boundary of the shaded region, is a
piece of a hyperbola and one can see that the tangent line to this
hyperbola at $(-1,0)$ makes a $30^{\circ}$ angle with the $x$-axis.

\begin{figure}
\begin{center}
$\underset{(a)\ \text{Triangles with $\alpha$, $\beta$ and
$\gamma$ as heights}
}{\epsfig{file=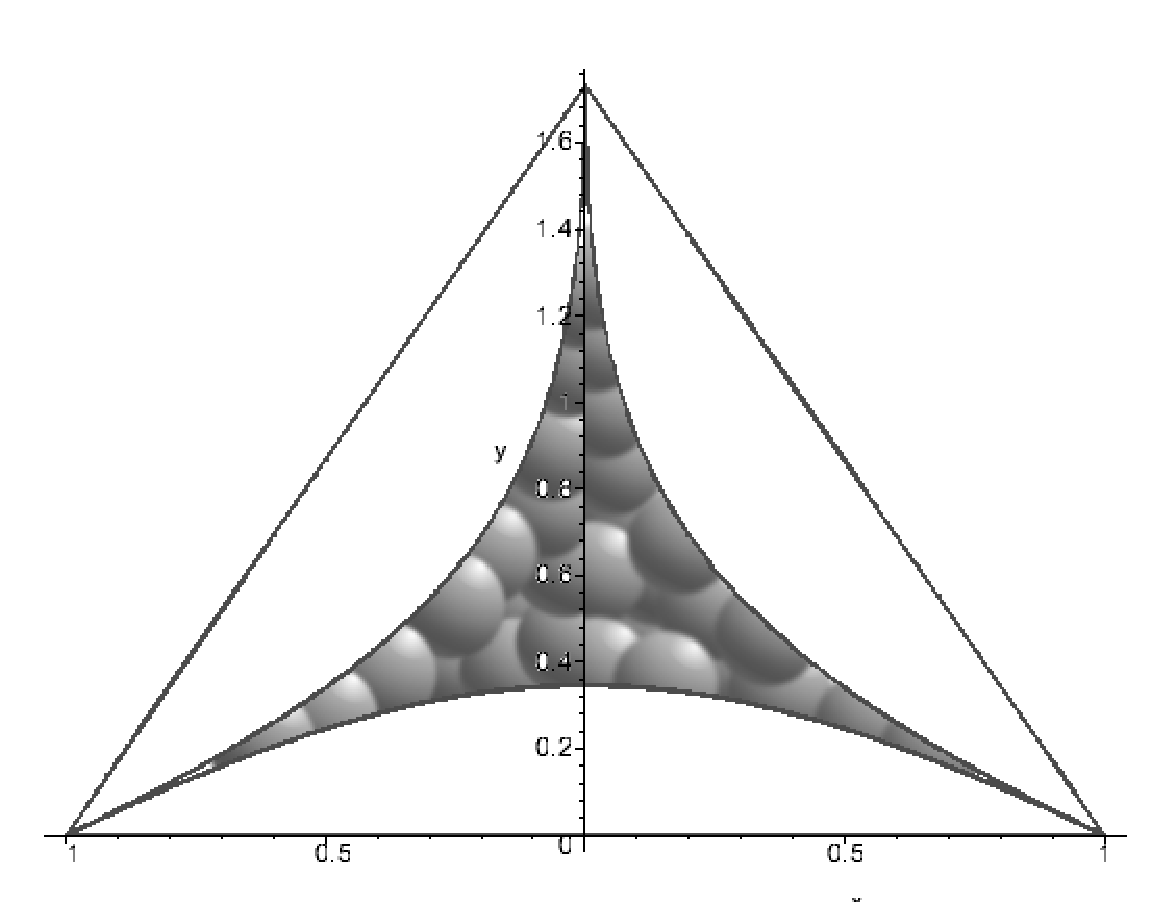,height=2.4in,width=2.4in}}$
$\underset{(b)\ \text{Acute triangles with $\alpha$, $\beta$ and
$\gamma$ as heights}
}{\epsfig{file=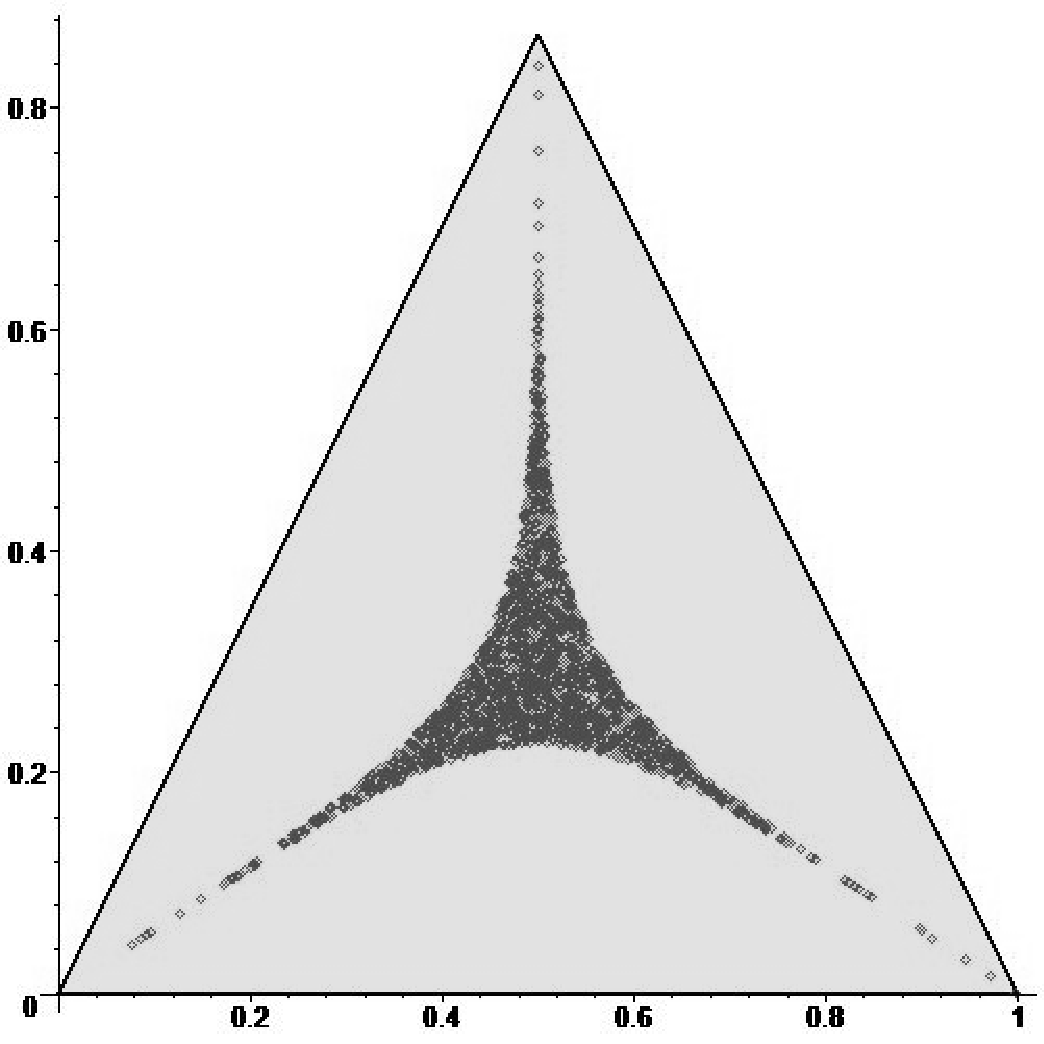,height=2.1in,width=2.2in}}$
\end{center}
\caption{ $\frac{Area(\mathcal R_1)}{Area(\mathcal R_2)})\approx
3$.} \label{figure4}
\end{figure}

This information is enough to conclude that the regions defined by
($\star$), ($\star\star$) and ($\star\star\star$) are disjoint.
Because of the symmetry of the problem, we can say that each such
region has an area of

$$A=2\int_0^1 \frac{3\sqrt{3}}{5}-\sqrt{\frac{3x^2}{5}+\frac{12}{25}}dx= 
2\sqrt{\frac{3}{5}}\int_0^1  \sqrt{x^2+\frac{4}{5}}dx,$$

\n which, after using formula (\ref{indefiniteint}) again, becomes

$$A=\frac{3\sqrt{3}}{5}-\frac{4}{5}\sqrt{\frac{3}{5}}\ln\left(\frac{3+\sqrt{5}}{2}\right).$$

\n Since the area of the triangle $ABC$ is $\sqrt{3}$, the
probability we are looking for is

$$P=1-3\frac{A}{\sqrt{3}}=\frac{12\sqrt{5}}{25}\ln\left(\frac{3+\sqrt{5}}{2}\right)-\frac{4}{5}\approx
0.2329814580.$$
\end{proof}

 For part (ii), we used Maple to compute  the probability that an acute triangle with
heights $\alpha$, $\beta$ and $\gamma$ exists, and found the
expression given in the statement of the theorem. We are not going
to include the derivation either since it is too cumbersome.
Experimentally there was a  fairly good match for the numerical
value given for the probability, and the ``picture" of the event
looks like the one in Figure~\ref{figure4}(b). It is very similar
to the one in Figure~\ref{figure4}(a), but with an area almost
three times smaller.

This gives a ratio
$\frac{P(obtuse)}{P(acute)}$ of about 2.008.

\vspace{0.1in}
\subsection{Radii of three mutually tangent circles and  tangent to the sides}\label{subsectiononrst}
\begin{figure}
\begin{center}
\[\underset{(a) \ r\ge s\ge t>0}{\epsfig{file=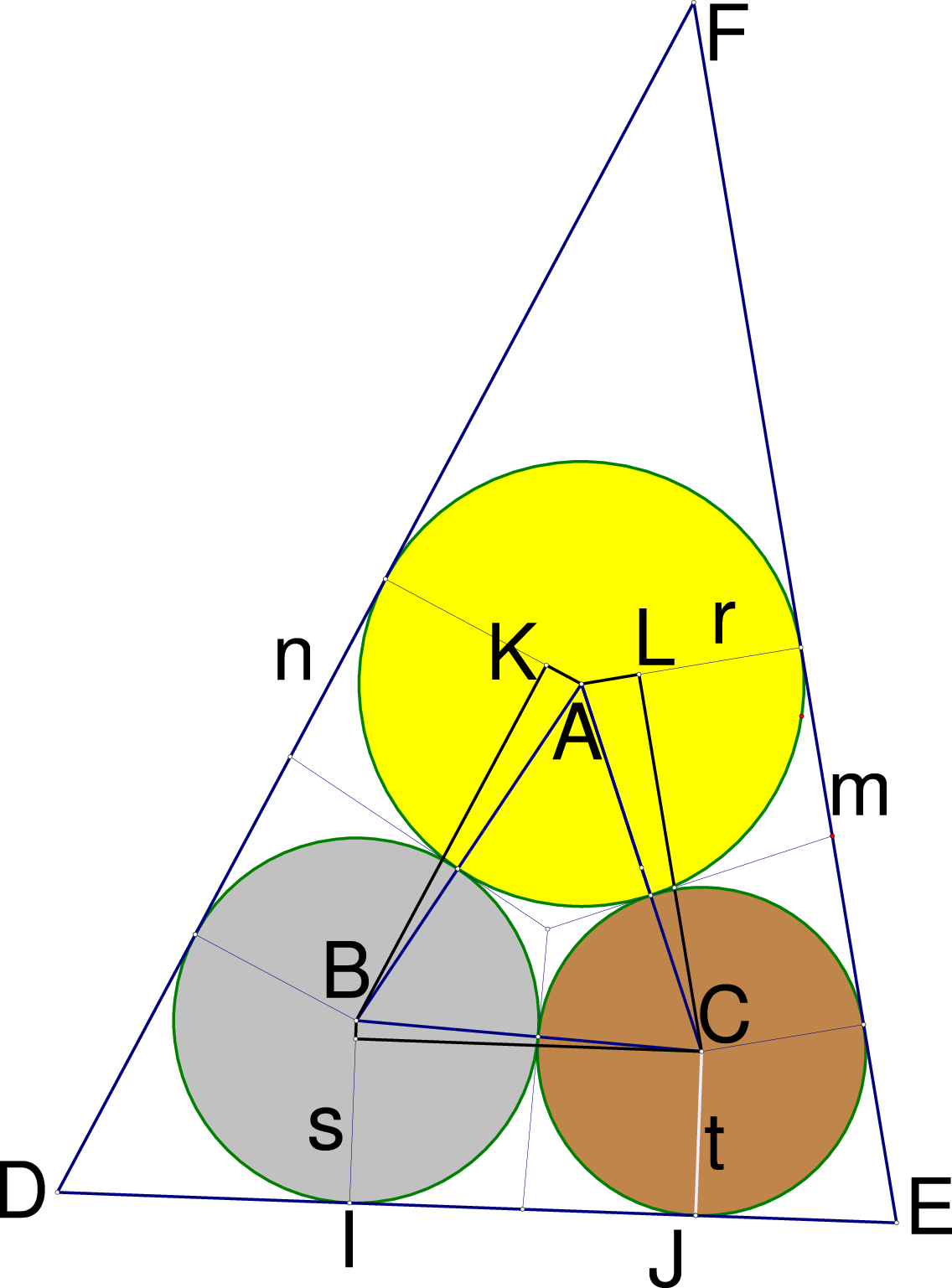,height=2in,width=1.6in}}\ \ \  \underset{(b)\ s<r,
\ s<4t}{\epsfig{file=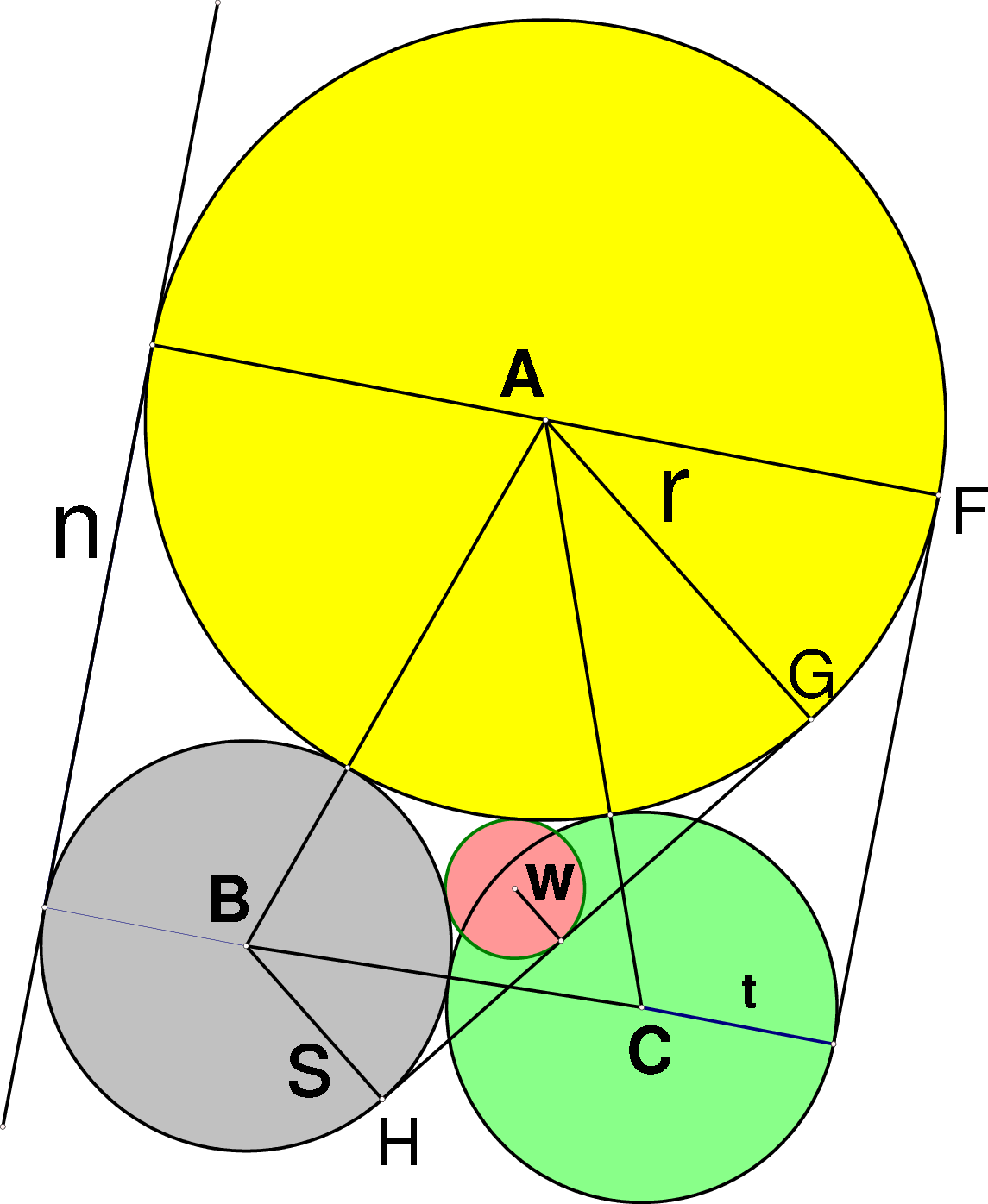,height=2in,width=1.6in}}
\]
\end{center}
\caption{\it Mutually tangent circles and their enclosing triangle}\label{figureitc}
\end{figure}
The following result has been inspired from a similar problem which appeared in Pi Mu Epsilon (\cite{wd}).
\begin{theorem}\label{thenewproblem} Under the hypothesis of the stick problem,
the probability that the three segments are the radii of three  circles tangent to
the sides of a triangle with each pair of these circles mutually externally tangent, is equal
to $\frac{5}{27}$.
\end{theorem}
\vspace{0.1in}

\n \proof   Let us denote by $r$, $s$ and $t$ the three lengths. We are beginning with the simple observation
 that a triangle with the sides $r+s$, $s+t$ and $t+r$ always exists. So, three circles externally tangent
 of radii $r$, $s$, and $t$ can be always constructed. Without loss of generality we may assume that
 $ r> s> t>0$ (the probability that two of the radii or all three to be equal is zero) and $t+s+r=\sqrt{3}$.
 To account for the other possible orders, we will multiply the probability we obtain in the end by $6$.
 We are denoting the center of the biggest circle by $A$, the next smaller circle's center by $B$ and
$C$ for the center of the smallest circle.
Then, the external tangent lines to each two of the circles exist. Out of all six tangent lines, we are clearly
looking here for those tangent lines which do no intersect  $\triangle  ABC$. Figure~\ref{figureitc}(b) suggests that it is
possible to have a triangle with its sides tangent (in exterior) to these circles but the interior of this triangle does not contain them.
So, excluding this situation, we have three
clearly defined tangent lines which do not intersect $\triangle ABC$ and have the potential to give the required triangle.
Basically, we need to characterize when these three tangent lines ``form" a triangle with the
circles in its interior as in Figure~\ref{figureitc}(a), in terms of $r$, $s$ and $t$.

Let us start with one of
the tangent lines, the one tangent to the smaller circles which
does not intersect the big circle. Let $I$ and $J$ denote  the two points
of tangency as in the Figure~\ref{figureitc}(a).

We consider a line through $c$ that is parallel  to $\overset{\leftrightarrow}{IJ}$
and form a rectangle and a right triangle by splitting
the trapezoid $BIJC$ into two parts. The Pythagorean Theorem gives
 the length of the tangent line segment to both of
the smaller circles as:
$IJ=\sqrt{(s+t)^2-(s-t)^2}=2\sqrt{st}$. Similarly, the tangent
line segment to the circles centered $A$ and $C$ has length
$2\sqrt{rt}$ and the third tangent segment is of length
$2\sqrt{rs}$.

Next, we let $m$ be the tangent line to the circle centered at $C$ and $A$ which does not intersect $\overline{AB}$.
We want to show that $\overset{\leftrightarrow}{IJ}$ intersects $m$, and we will denote the point of their intersection
(except when these lines coincide) by $E$. The order between $r$, $s$ and $t$ tells us that the angle $\angle
ACB$ is the biggest angle of $\triangle  ABC$ and so it is more
than $60^{\circ}$. The angle between the tangent lines $m$ and $\overset{\leftrightarrow}{IJ}$, say
$\omega$, is then more than $60^{\circ}$ and less than
$180^{\circ}+2(90)^{\circ}=360^{\circ}$ (including the reflex
angle possibility). Since the case $\omega=180^{\circ}$ means that the two tangent lines coincide, these tangent lines always have  a point of intersection.
  In order to have a triangle $DEF$ containing in its interior the
three circles, we need to limit $\omega$ to less than
$180^{\circ}$.

Let us observe (see Figure~\ref{figureitc}(b)), that $\omega\ge 180^{\circ}$ if and only if $t$ is
smaller than the radius $w$ of a circle tangent to the bigger
circles and their common tangent line. By what we have observed
earlier, the radius $x$ must satisfy $2\sqrt{sw}+2\sqrt{rw}=2\sqrt{rs}$. This means that
$w=rs/(\sqrt{s}+\sqrt{r})^2=s/(\sqrt{\frac{s}{r}}+1)^2>
s/4$. So, the first restriction we need to have on these
numbers is that $t>w$, which attracts

\begin{equation}\label{feq}
t>\frac{rs}{(\sqrt{s}+\sqrt{r})^2}>\frac{s}{4}, \ \ \text{or}\ \
r<\frac{st}{(\sqrt{s}-\sqrt{t})^2} .
\end{equation}
We observe that the third tangent line, the one which does not intersect $\triangle ABC$,
denoted in Figure~\ref{figureitc}(a) by  $n$, is
ensured by (\ref{feq}) to intersect
$\overset{\leftrightarrow}{IJ}$ so we will let $D$ be the point of
their intersection. Let $L$ be the point of intersection of the parallel
to $m$ through $C$ with the radius corresponding to the tangency
point on $m$  and similarly on the other side we let $K$ be the analogous
point (Figure~\ref{figureitc}(a)).

Finally, to ensure that $m$ and $n$ intersect, on the same side of
$\overset{\leftrightarrow}{DE}$ as the circles, we need to have

$$m(\angle KBA)+m(\angle ABC)+m(\angle BCA)+m(\angle
ACL)<180^{\circ},$$

\n by the original Euclid's fifth postulate. This is equivalent
to

$$\arcsin\left(\frac{r-s}{r+s}\right)+\arcsin\left(\frac{r-t}{r+t}\right)<m(\angle BAC).$$

\n Because $u\mapsto \cos u $ is a decreasing function for $u\in
[0,180^{\circ}]$, using the law of cosines in the triangle $ABC$
and the formula $\cos(\alpha+\beta)=\cos \alpha\cos
\beta-\sin\alpha\sin \beta$, this last inequality translates into

$$\frac{(r+s)^2+(r+t)^2-(s+t)^2}{2(r+s)(r+t)}<\frac{2\sqrt{rs}}{r+s}\cdot \frac{2\sqrt{rt}}{r+t}-\frac{(r-s)(r-t)}{(r+s)(r+t)}.$$

\n After some algebra, one can reduce this to

\begin{equation}\label{seq}
r<2\sqrt{st}.
\end{equation}

\n Let us observe that
$2\sqrt{st}<st/(\sqrt{s}-\sqrt{t})^2$ is equivalent to
$2s+2t-5\sqrt{st}<0$ or $(2\sqrt{\frac{s}{t}}-1)(
\sqrt{\frac{s}{t}}-2)<0$. This is true under the necessary
condition $s<4t$. So, the existence of an encompassing triangle
around the three circles of radii $r$, $s$, $t$ satisfying $t<s<r$
is given by (\ref{seq}), and $s<4t$.

Without loss of generality, let us employ our model, in such a way that
$r=\alpha=y$ and $s=\beta=\frac{\sqrt{3}(1+x)-y}{2}$, and
$t=\gamma=\frac{\sqrt{3}(1-x)-y}{2}$.

\begin{figure}
\[
\underset{(a) \ Computer\ simulation }{\epsfig{file=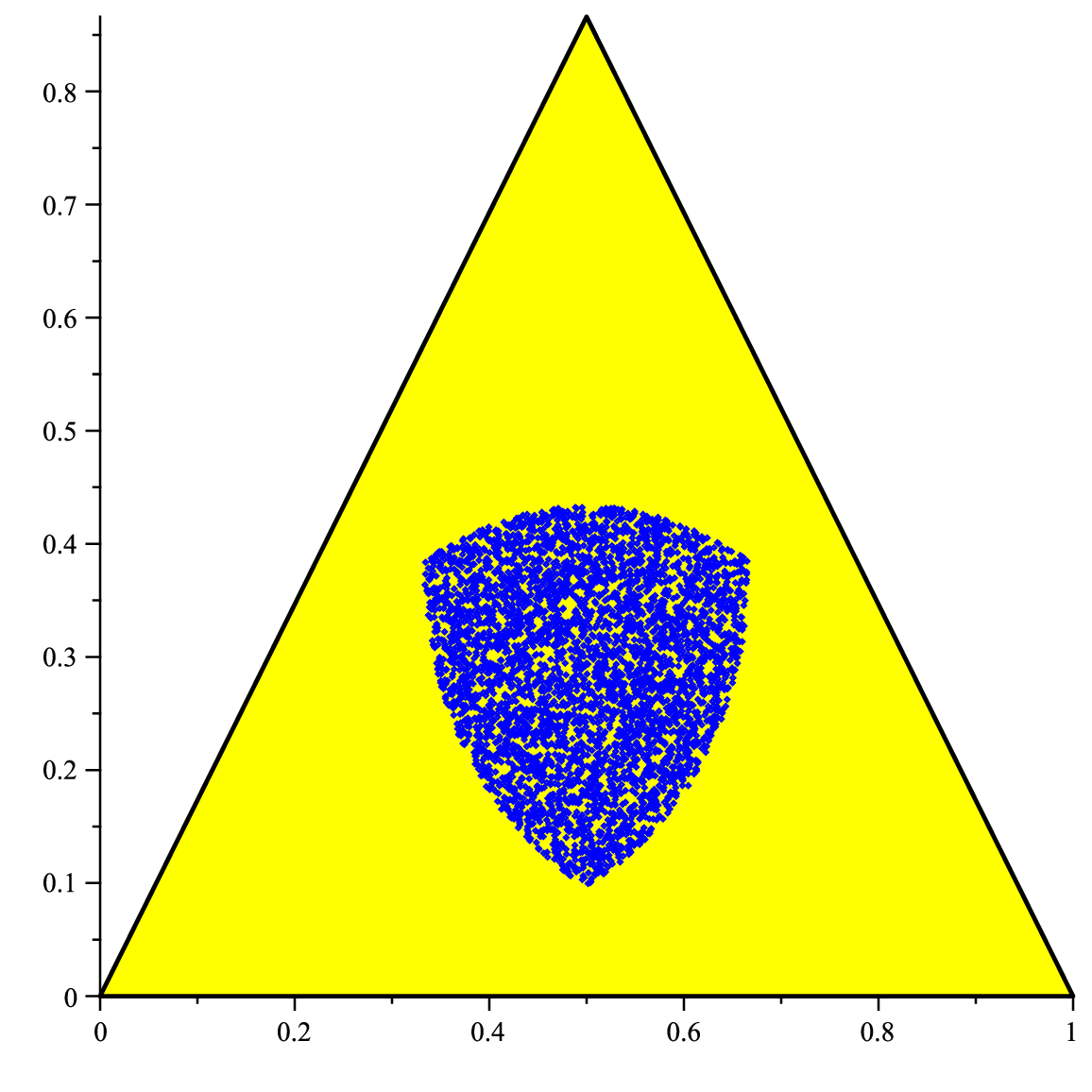,height=1.8in,width=2in}}\ \ \
\underset{(b)\
\ 0<x<\frac{1}{3},\ y<\frac{\sqrt{3}(1-x^2)}{2},\
y>\frac{1+x}{\sqrt{3}}
}{\epsfig{file=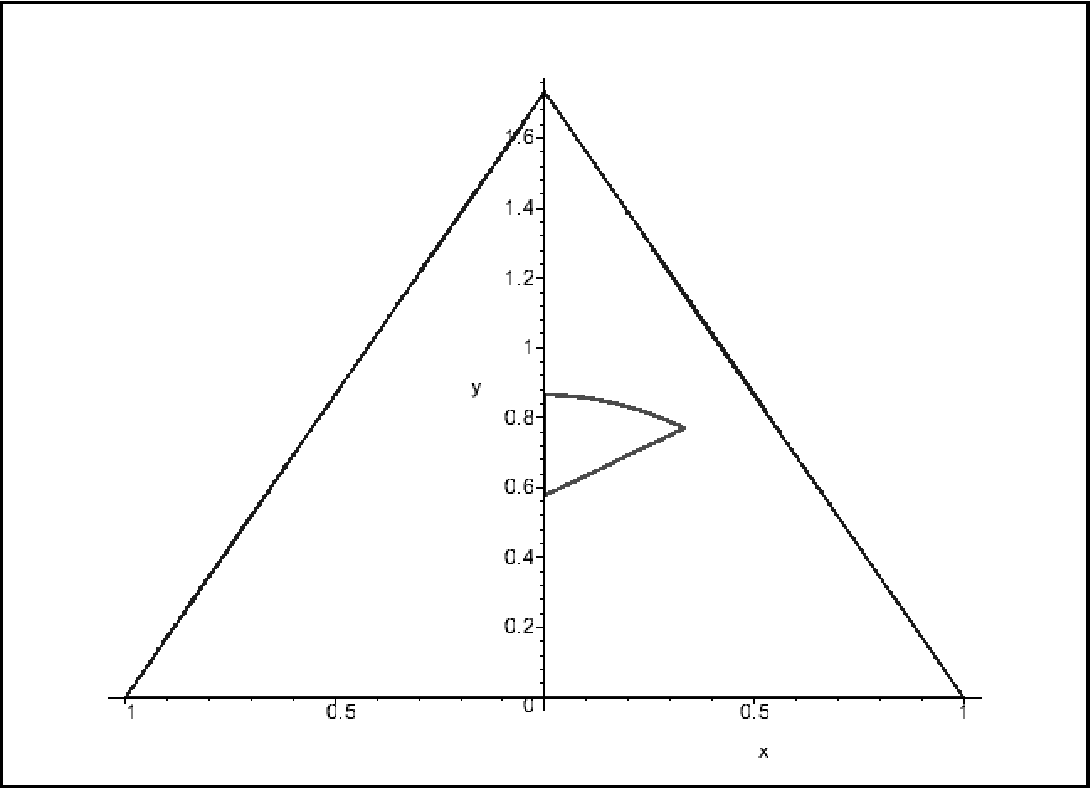,height=2in,width=2in}}
\]
\caption{\it $A(1,0)$, $B(-1,0$) and $C(0,\sqrt{3})$, $ON=t$,
$OM=s$, $OP=r$ }\label{fig9tic}
\end{figure}

The condition $t<s$ is equivalent to $0<x$ and the inequality
$s<r$ implies $y>(1+x)/\sqrt{3}$ (Figure~\ref{fig9tic}(b) ). The restriction
(\ref{seq}) is the same as $y<\frac{\sqrt{3}}{2}(1-x^2)$. Also,
let us observe that the last restriction $s<4t$ is equivalent to
$y<\frac{3-5x}{\sqrt{3}}$. It turns out that
$\frac{\sqrt{3}}{2}(1-x^2)<\frac{3-5x}{\sqrt{3}}$ is satisfied if
$x<1/3$ which is a restriction already given by the the other
inequalities we have (Figure~\ref{fig9tic}(b)). This gives

$$
\begin{array}{c}
\ds P=\frac{6}{\sqrt{3}}\int_0^{\frac{1}{3}}
\left[\frac{\sqrt{3}}{2}(1-x^2)-\frac{1+x}{\sqrt{3}}\right]dx=\\ \\
\ds \int_0^{\frac{1}{3}}(1-2x-3x^2)dx=(x-x^2-x^3)|_0^{1/3}=\frac{5}{27}.\ \ \ \bsq
\end{array}
$$

We note  that we have actually obtained the following:
{\it three circles of positive radii, $r,s,t>0$, allow the existence of a triangle as in Figure~\ref{figureitc}(a), if and only if,
$\max(r,s,t)^3< 4rst$.}

\vspace{0.2in}

For the case of acute triangles, one can check that it is necessary and sufficient that

$$\begin{array}{c}
2\sqrt{(r+s+t)rst}\le \\ \\ \min\{ 4t\sqrt{rs}-|(r-t)(s-t)|, 4r\sqrt{st}-|(r-t)(s-r)|,4s\sqrt{rt}-|(r-s)(s-t)|\}
\end{array}$$

\n condition which allowed us to compute the probability experimentally: $P(acute)\approx 0.047845$. This  gives
a ratio between the obtuse and the acute cases which is close to $3$  (see Figure~\ref{Figure9}(a) for the corresponding shape).

\begin{figure}
\[
\underset{(a)\  \text{Acute  case $r$, $s$, $t$ } }{\epsfig{file=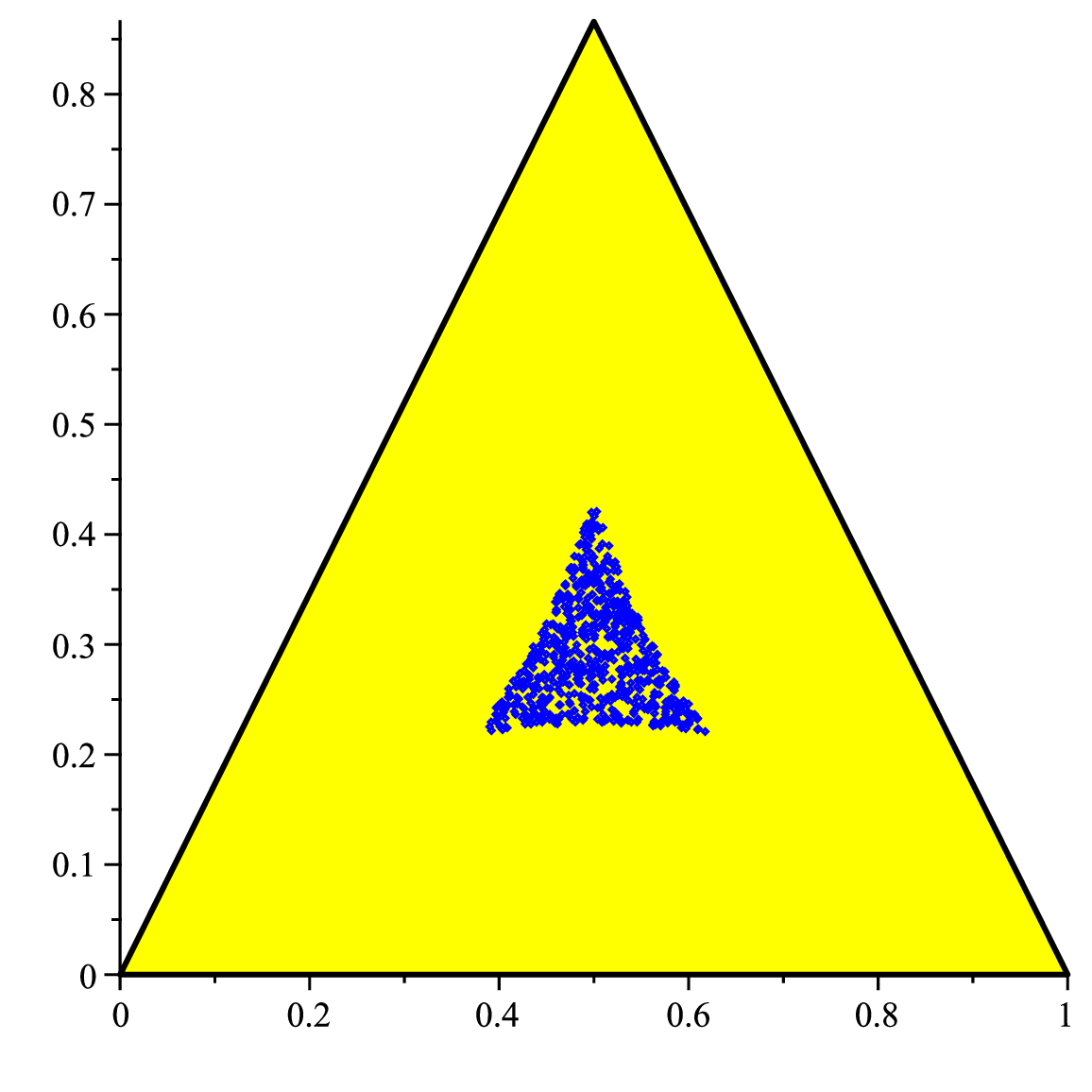,height=1.8in,width=2in}}\ \ \ \ \
\underset{(b)\ \text{Angle bisectors region for acute triangle};\  P\approx 0.1195.
}{\epsfig{file=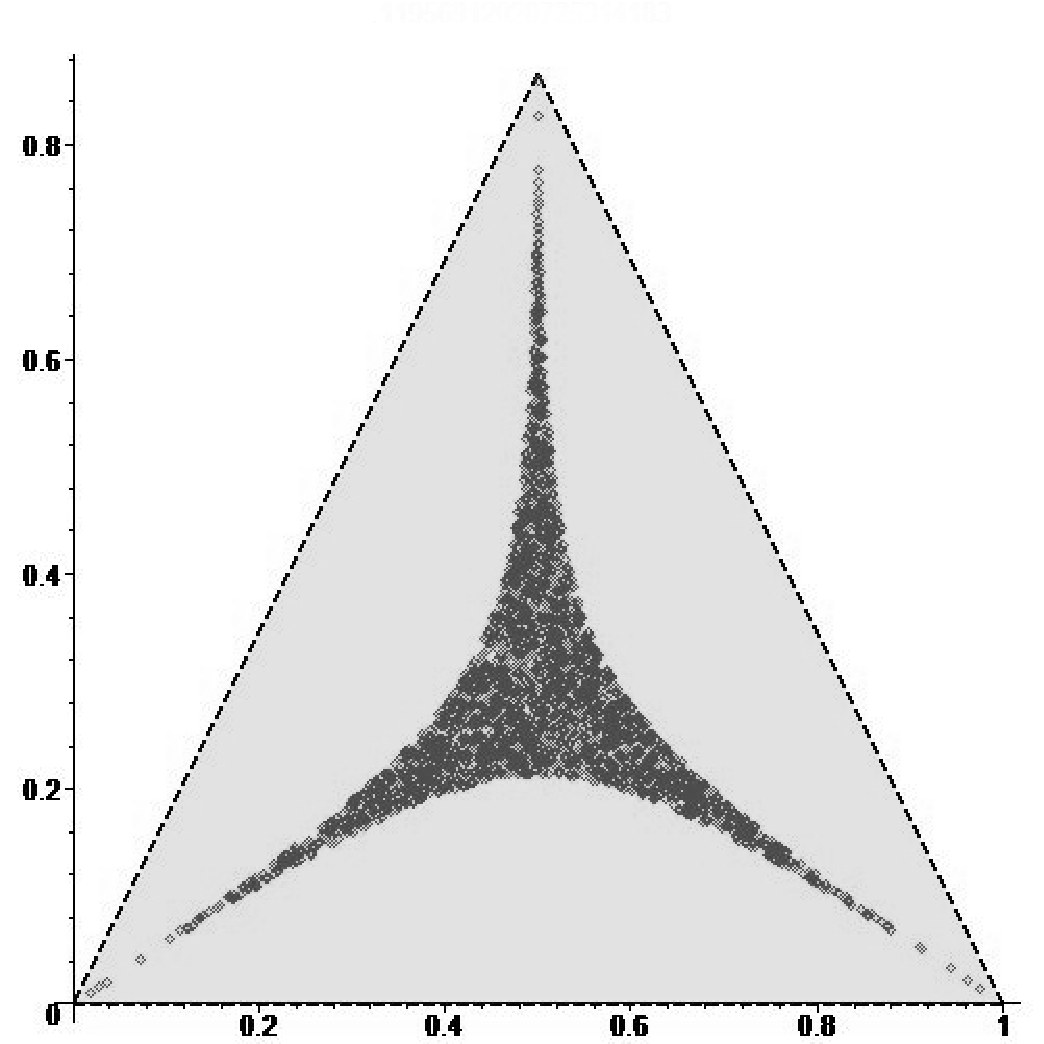, height=1.9in,width=1.9in}}
\]
\caption{\it Special\ Cases }\label{Figure9}
\end{figure}

\n\subsection{Radii of excircles.}

In Figure~\ref{figure5} (a) we
show the three excircles of a triangle. Let $r_a$, $r_b$, and
$r_c$ denote the radii of the excircles of a triangle $ABC$ that
are tangent to the sides $\overline{BC}$,  $\overline{AC}$, and
$\overline{AB}$ respectively.

\begin{figure}
\begin{center}
$\underset{(a)\ \text{Excircles}
}{\epsfig{file=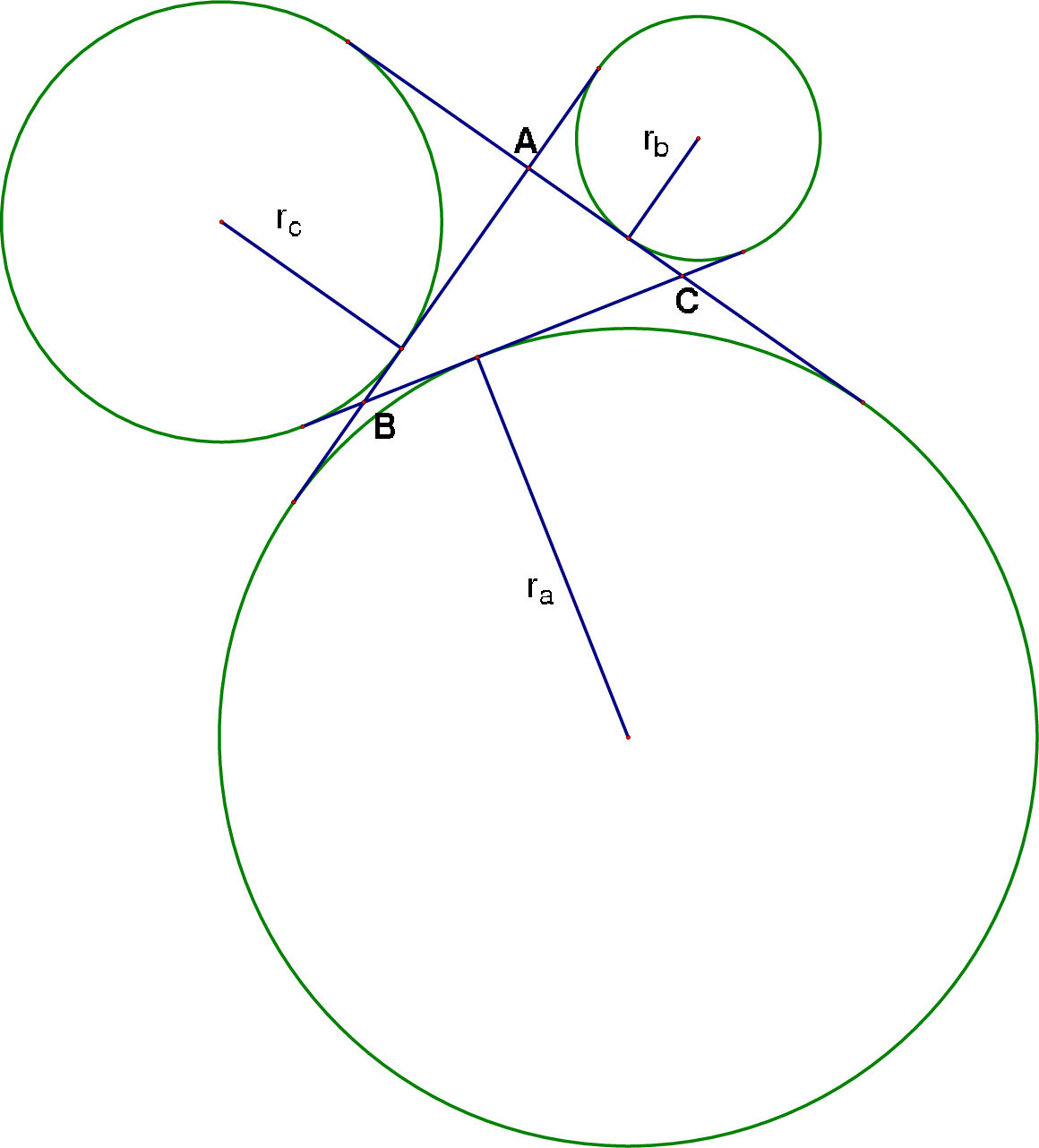,height=2in,width=2in}}$\ \ \
$\underset{(b)\ \text{Acute triangles with $\alpha$, $\beta$ and
$\gamma$ as radii of excircles}
}{\epsfig{file=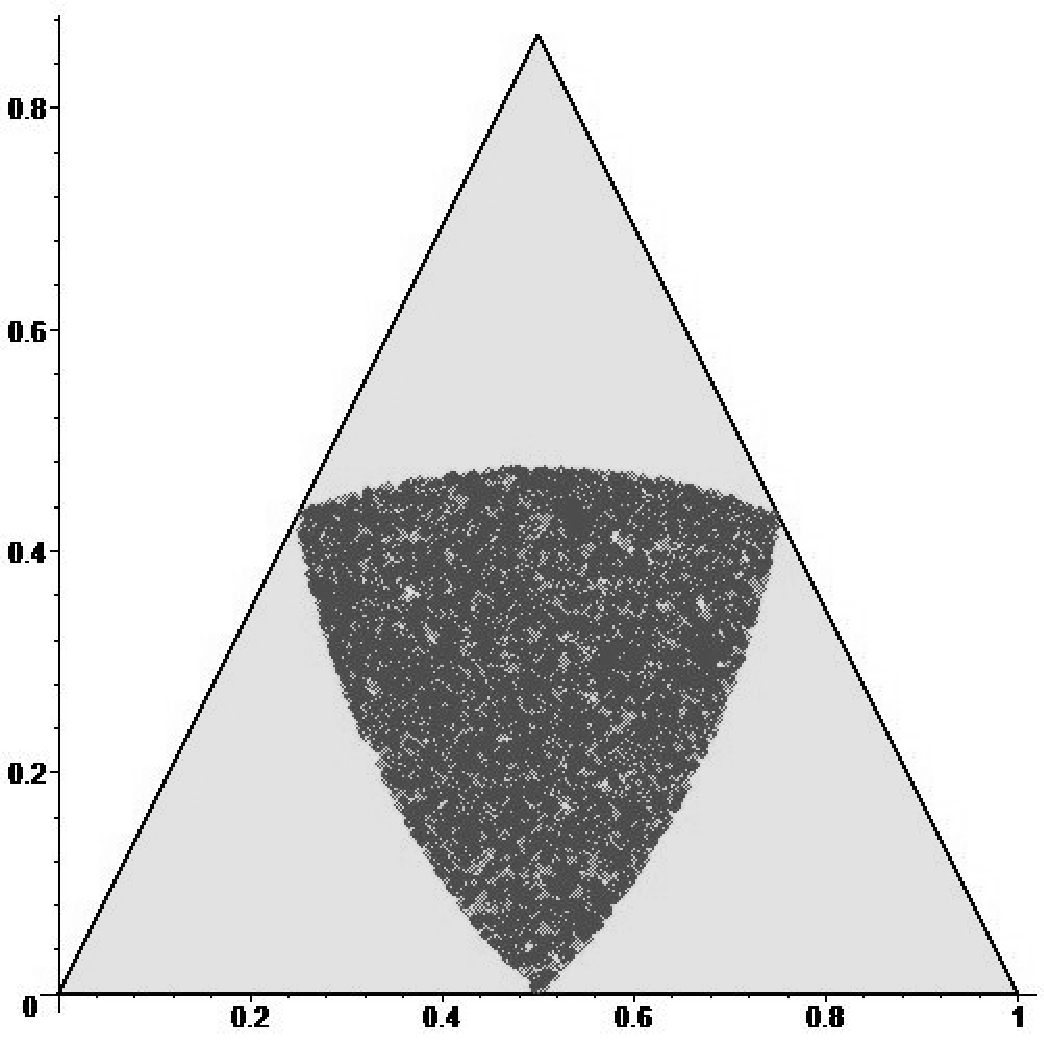,height=2in,width=2.1in}}$
\end{center}
\caption{Circles tangent to the sides}
\label{figure5}
\end{figure}

\begin{theorem}\label{radiiofexcircles}
If $u$, $v$, and $w$ are greater than $0$ then there is a unique
triangle such that $r_a = u$, $r_b = v$, and $r_c = w$. Moreover,
this triangle is acute if and only if $uv + vw + wu > \max \{ u^2,
v^2, w^2\}$.
\end{theorem}
\begin{proof}
In any triangle we have
\[
r_a = \frac{2S}{b + c - a}, \ \ \ r_b = \frac{2S}{a + c - b},\ \ \
\text{and} \ \ \ r_c = \frac{2S}{a + b - c}.
\]

\n Hence, we need to show that the system
\[
\frac{2S}{b + c - a} = u, \ \frac{2S}{a + c - b} = v,\ \frac{2S}{a
+ b - c} = w
\]
or
\[
b + c - a = \frac{2S}{u}, \ a + c - b = \frac{2S}{v},\ a + b - c =
\frac{2S}{w},
\]

\n has a unique solution for $a$, $b$, and $c$. Adding the three
equations above gives $a+b+c= 2S \frac{uv+vw+wu}{uvw}$.
Multiplying this last equality by the previous three, we obtain
\[
(a + b + c)(b + c - a)(a + c - b)(a + b - c) = 16 S^4
\frac{uv+vw+wu}{u^2v^2w^2}
\]
\n from which, using Heron's formula, we get
\[
16 S^2 = 16 S^4 \frac{uv+vw+wu}{u^2v^2w^2} \Rightarrow S =
\frac{uvw}{\sqrt{uv+vw+wu}}.
\]
This changes the previous system for $a$, $b$, and $c$ into
\[
a + b - c = \frac{2 uv}{\sqrt{uv+vw+wu}},\  a + c - b =
\frac{2uw}{\sqrt{uv+vw+wu}},\ \text{and} \]

\[\ b + c - a = \frac{2vw}{\sqrt{uv+vw+wu}}
\]
which, by adding pairs of these equalities, provides the solutions
\[
a = \frac{uv+uw}{\sqrt{uv+vw+wu}}, \ b =
\frac{uv+vw}{\sqrt{uv+vw+wu}}, \ \text{and}\ c =
\frac{uw+vw}{\sqrt{uv+vw+wu}}.
\]
It it clear that these solutions satisfy the triangle inequality,
so this proves the existence and uniqueness stated in the first
part of our theorem.

\par Next, the triangle is acute if and only if
\[
a^2 + b^2 > c^2, \quad a^2 + c^2 > b^2 \quad \textup{and} \quad
b^2 + c^2 > a^2
\]
which is equivalent to
\[
uv + vw + u w > u^2, \quad u v + v w + w u > v^2 \quad
\textup{and} \quad  u v + v w + w u > w^2.
\]
The last three inequalities can be put together as in our
statement of the theorem. \end{proof}

\begin{corollary}\label{exteriorcircles}  {\it Under the hypothesis of the broken stick problem,
the probability that the three segments are the radii of the
excircles of a triangle is equal to 1. Moreover, the probability
that this triangle be an acute triangle is:
$$ P=\frac{24\sqrt{7}}{49}\arcsin(\frac{\sqrt{14}}{8})-\frac{2}{7}\approx
0.3449830931 $$}
\end{corollary}

\begin{proof} By the second part of Theorem~\ref{radiiofexcircles}
we need to calculate the area of the region characterized by
$$\alpha\beta+\beta\gamma+\gamma\alpha>max(\alpha^2,\beta^2,\gamma^2).$$

Let us look at one inequality, say
$\alpha\beta+\beta\gamma+\gamma\alpha>\alpha^2$. With the
substitutions from (\ref{definitionofalphabetagamma}), this
becomes
$$|7y-\sqrt{3}|< \sqrt{3(8-7x^2)}.$$ If we plot all the
corresponding ellipses, we get a picture as in
Figure~\ref{figure5} (b) (the boundary of the shaded region).  It
is easy to see that the ellipse above cuts the sides of the
equilateral triangle exactly in half. Hence the probability we are
looking for is equal to

$$P=\frac{1}{\sqrt{3}}
\left( \frac{\sqrt{3}}{4}+
3\int_{-\frac{1}{2}}^{\frac{1}{2}}\frac{\sqrt{3}+\sqrt{3(8-7x^2)}}{7}-\frac{\sqrt{3}}{2}dx\right)$$

or

$$P=\frac{6}{7}\int_{-\frac{1}{2}}^{\frac{1}{2}}\sqrt{8-7x^2}dx-\frac{23}{28}.$$

Finally, this gives
$P=\frac{24\sqrt{7}}{49}\arcsin(\frac{\sqrt{14}}{8})-\frac{2}{7}\approx
0.3449830931$.\end{proof}
\section{Special Cases}

Although there are possibly other problems in which exact answers
could be found we move on to some other surprising results.

\subsection{Angle Bisectors}\label{AngleBisectors}

Due to a paper of Mironescu and Panaitopol \cite{mp}, we know that the
probability that a triangle $ABC$ exists so that $\alpha=w_a$,
$\beta=w_b$ and $\gamma=w_c$, is 1. It is important to mention
that the existence problem involved in this result had been open
since 1875. The proof in  \cite{mp} is based on  Brouwer's fixed point
theorem. We used the contractive map described in this work
and built a Maple program that tested the condition of obtaining
an acute triangle. The region and the frequency obtained from
using 50,000 randomly selected points in our model (generated by
picking at random with uniform distribution from 500 points on the
two sides as described in the Introduction), with a stopping error
for iterations of 0.0001 are shown in the Figure~\ref{Figure9}(b).
We tried to determine the equations of the boundary for the
region in Figure~\ref{Figure9}(b) which corresponds to right
triangles. Our direct approach was less successful in this case,
since the equation of the boundary involved the two angle
bisectors $u$ and $v$ (the angle bisector from the right angle is
assumed to be equal to one, $w=1$), and a root $r$  of the sixth
degree equation (in $Z$)

$$8Z^6u^2-8\sqrt{2}u^2Z^5-8(u^2-1)Z^4+8\sqrt{2}u^2Z^3-8Z^2+1=0$$

\n in a twenty six term polynomial of degree ten (in $\mathbb
Z[\sqrt{2}][u,v,r]$).

\subsection{Altitude, angle bisector and a median}\label{aabm}

\begin{theorem}\label{theoremabm}
If $0 < u < v < w$ then there is a triangle such that the
altitude, the angle  bisector, and the median from one of the
vertices of the triangle equal $u$, $v$, and $w$ respectively.
Moreover, this triangle is acute if and only if
\begin{equation}\label{abm}
\frac{u\sqrt{v^4-3u^2(v^2-u^2)}}{2u^2-v^2}<
w<\frac{uv^2}{2u^2-v^2}.
\end{equation}
\end{theorem}
\begin{proof}
Let $\overline{AD}$ be a segment of length $u$. We construct the
perpendicular at $D$ on $\overline{AD}$. Let $N$ be a point on
this perpendicular such that the length of $\overline{AN}$ is $v$
and $M$ a point on the same perpendicular such that
$\overline{AM}$ has length $w$ and $N$ is between $D$ and $M$ (as
in Figure~\ref{Figure7} (b)).

\begin{figure}
\begin{center}
$\underset{(a)\ Orthocenter,\ \{\alpha,\ \beta,\ \gamma\}=\{\ HD,
\ HE,\ HF\}}{\epsfig{file=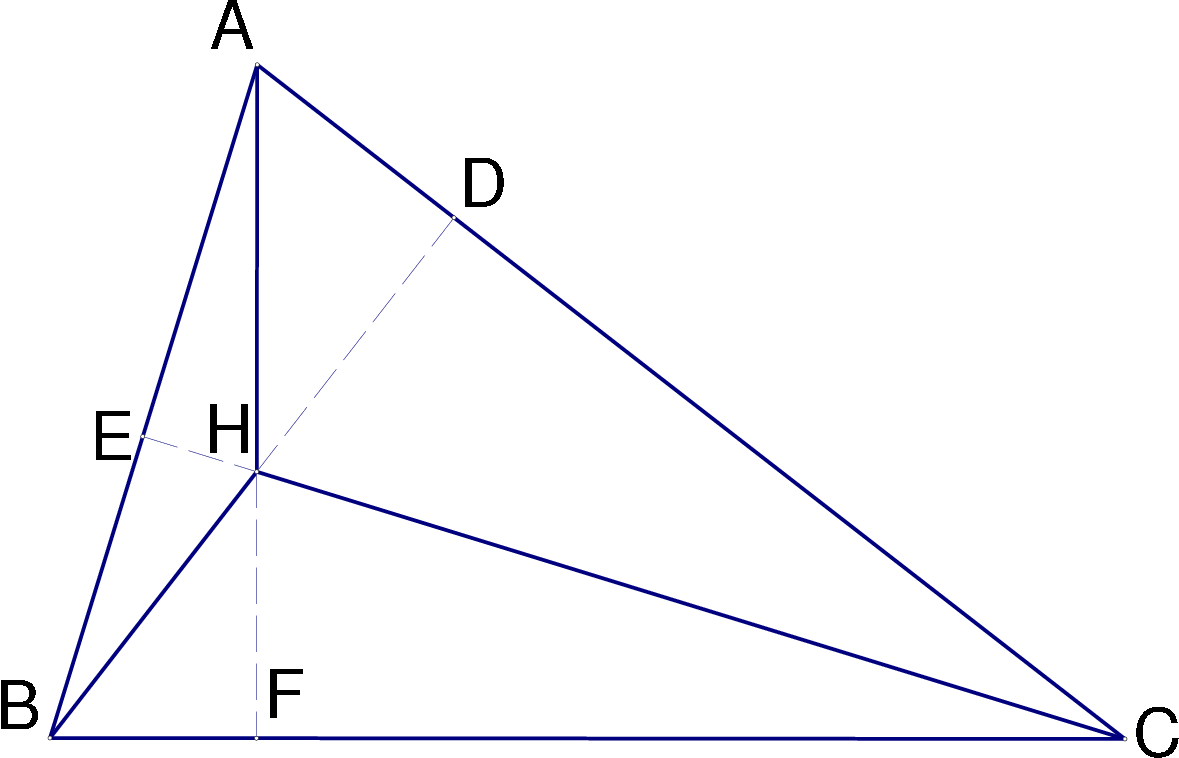,height=1.7in,width=2.5in}}$\
\ \ \ $\underset{(b)\ \{\alpha,\ \beta,\ \gamma\}=\{\ h_a, \
w_a,\ m_a\} }{
\epsfig{file=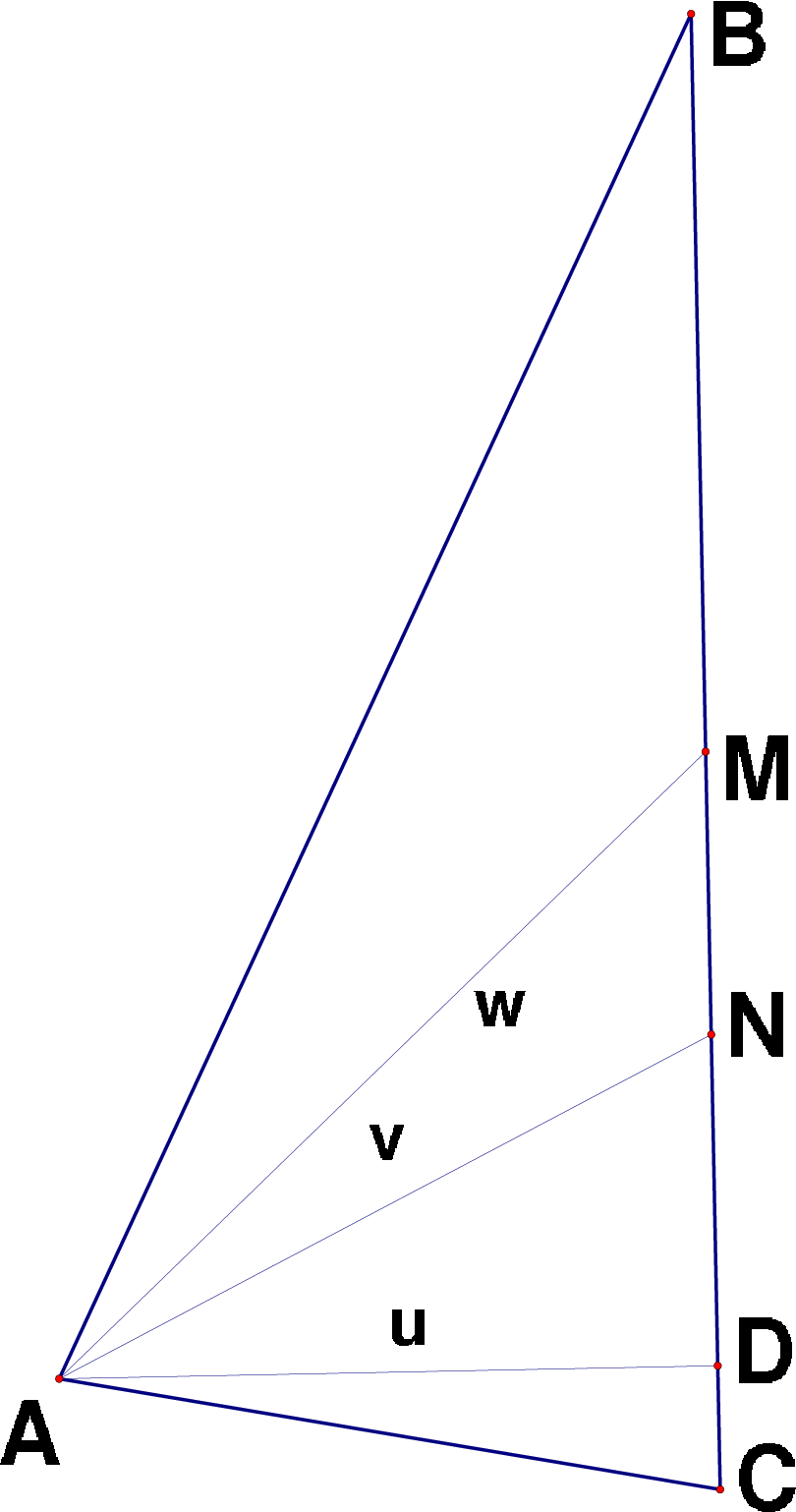,height=2.1in,width=1.6in}}$
\end{center}
\caption{Triangle\ ABC, $h_a=u$,\ $w_a=v$, and \ $m_a=w$.}
\label{Figure7}
\end{figure}

Our goal is to find two points $B$ and $C$ on
$\overset{\leftrightarrow}{ DM}$ such that in the triangle $ ABC$,
$\overline{AD}$  is an altitude, $\overline{AN}$ is an angle
bisector, and $\overline{AM}$ is a median. Let $B$ be a point on
$\overset{\leftrightarrow}{DM}$ such that $M$ is in between $N$
and $B$, and $C$ a point on $\overset{\leftrightarrow}{ DM}$ such
that $M$ is the midpoint of $\overline{BC}$. In the triangle $
ABC$ just obtained, it is clear that $\overline{AD}$ is an
altitude and $\overline{AM}$ is a median. We need to show that we
can move $B$ and $C$ to such positions that will also make
$\overline{AN}$ an angle bisector. To simplify the computation we
will denote the length of $\overline{BM}$ and $\overline{CM}$ by
$t$, the length of $\overline{DM}$ by $\delta$ and the length of
$\overline{MN}$ by $\omega$. Then
\[
AB^2 = u^2 + (t + \delta)^2 \qquad \textup{and} \qquad AC^2 = u^2
+ (t - \delta)^2
\]
while, by the Angle Bisector Theorem,
\[
\frac{AB}{AC} = \frac{BN}{NC} = \frac{t + \omega}{t - \omega}.
\]

\n Thus, we get the equation
\[
\frac{(t + \omega)^2}{(t - \omega)^2} = \frac{u^2 + (t +
\delta)^2}{u^2 + (t - \delta)^2}
\]
\n which, after some simplifications, becomes
\[
4 t \omega u^2 = 4 t (\delta - \omega) (t^2 - \omega \delta) \iff
t^2 = \omega \delta + \frac{\omega u^2}{\delta - \omega}.
\]
This shows that there is a unique solution to this problem.

For the second part of the theorem, note that we always have $AB^2
+ BC^2 - AC^2 > 0$ since $AB > AC$. In other words, by
construction we automatically have   $\angle B<90^{\circ}$. Angle
$\angle C<90^{\circ}$ if and only if $AC^2 + BC^2 - AB^2 > 0$.
This is the same as $t > \delta$ or $\omega u^2 - \delta (\delta -
\omega)^2 > 0$.

Looking back, where we introduced the notation, we see that
\[
\delta = \sqrt{w^2 - u^2} \qquad \textup{and} \qquad \omega =
\sqrt{w^2 - u^2} - \sqrt{v^2 - u^2}.
\]

\n Using these expressions  the above inequality  becomes

$$(2u^2-v^2)\sqrt{w^2-u^2}-u^2\sqrt{v^2-u^2}>0.$$

\n It is clear that if $2u^2-v^2\le 0$ the above  inequality is
false. So, we need to have $2u^2>v^2$ and under this assumption
the inequality above is the same as

$$(2u^2-v^2)^2w^2>u^2\left(v^4-3u^2(v^2-u^2)\right).$$

\n This shows that the first inequality in (\ref{abm}) must be
true if the triangle $ABC$ is an acute triangle.

Angle $ A<90^{\circ}$ if and only if $AB^2 + AC^2 - BC^2
> 0$, and because the length of BC is $2 t$, this is simply
equivalent to $u^2 + \delta^2 - t^2 > 0$. After substitution for
$t$, this becomes

\[ \delta (\delta - \omega)^2 + u^2 (\delta - 2
\omega) > 0.
\]

\n By substitution as before, this becomes

\[
\sqrt{w^2 - u^2} (v^2 - u^2) +u^2 (2 \sqrt{v^2 - u^2} - \sqrt{w^2
- u^2}) > 0 \iff \]
\[2 u^2 \sqrt{v^2
- u^2}>(2u^2-v^2) \sqrt{w^2 - u^2}.
\]
 So, since we may assume $v^2< 2u^2$,
the inequality above becomes the same as

$$w<\frac{uv^2}{2v^2-w^2}.$$
We notice that the condition $2u^2>v^2$ is implicitly assumed true
if (\ref{abm}) is satisfied.  Hence, we have shown the necessity
and the sufficiency of the conditions (\ref{abm}) in the theorem.
\end{proof}

\n {\bf Remark:} One can check that
$\frac{u\sqrt{v^4-3u^2(v^2-u^2)}}{2u^2-v^2}>b$ is equivalent to
$v^2<3u^2$ and so the restrictions (\ref{abm}) are always
non-trivial.

\vspace{0.1in}

\begin{corollary}\label{probabilityabm}
Assuming that $u$, $v$ and $w$ in Theorem~\ref{theoremabm} are the
ordered triple given by a broken stick, the probability, that the
triangle insured by Theorem~\ref{theoremabm} is acute, equals

$$m\int_{\frac{m}{2\sqrt{2}+1}}^{m/3} (1-g(s))^2- \frac{4(7s^2+2ms-3)^2}{(3ms^3+5s^2+3ms-3)^2} ds\approx
0.04223393591,$$

\n where $m=\sqrt{3}$ and $g$ is defined by
$g(t)=\frac{(t-m)A(t)+2t\sqrt{B(t)}} {(t+m)A(t)+2t\sqrt{B(t)}}$
with $$\begin{cases}
A(t)=7t^2+2mt-3\\ \\
B(t)=37t^4+20mt^3-18t^2-12mt+9.
\end{cases}$$
\end{corollary}

We will include just the idea of proof for this corollary because
the calculations are very cumbersome. However, one can check them
with a symbolic algebra program such as Maple or Mathematica.

Here our  idea is basically the same as in all of the previous
problems. Depending of the order of the $\alpha$, $\beta$ and
$\gamma$, there are six possible regions in  our model. We are
going to pick one of them, say, $\alpha<\gamma<\beta$ and then the
values of $u$, $v$ and $w$ are given as in the Introduction (see
Figure~\ref{Figure1}~(a)), in terms of $x$ and $y$, by $u=\alpha$,
$w=\beta$,  and $v=\gamma$ defined in
(\ref{definitionofalphabetagamma}).

\begin{figure}
\begin{center}
\epsfig{file=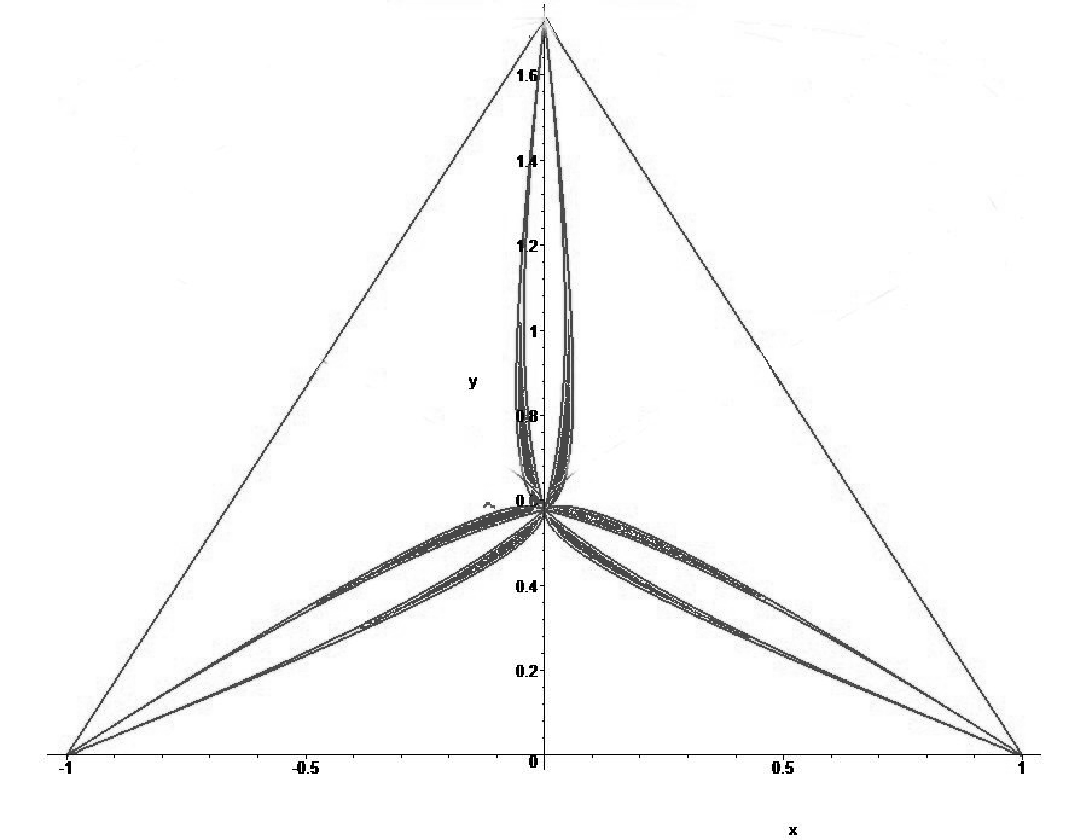,height=2.5in,width=3in}
\end{center}
\caption{Region \ defining\ the \ acute\ triangle.}
\label{Figure8}
\end{figure}

Taking into account symmetries, the two inequalities in
(\ref{abm}) define the region depicted in Figure~\ref{Figure8}. We
are going to concentrate only on one sixth of the picture. The
conditions $0<u<v<w$ are equivalent to $x>0$, $y>0$, and
$y<\frac{\sqrt{3}}{3}(1-x)$.

In order to obtain something that we can integrate we need to
parameterize the two resulting curves from (\ref{abm}). The idea
is to make the substitution $y=t(1-x)$ which will considerably
simplify the equations of the two curves. This is a standard
procedure of rationalizing a curve if one knows a point with
integer coordinates on it (see \cite{s}). The inequality
$w(2u^2-v^2)<uv^2$ turns into

$$x(9t^3+5t^2m+9t-3m)-9t^3+9t^2m +3t-3m<0$$

\n or

$$x<f(t):=\frac{9t^3-9t^2m -3t+3m}{9t^3+5t^2m+9t-3m}.$$

Let us  observe that $t$ is less than $\frac{m}{3}$. One the other
hand it is obvious that we need to have $v<u\sqrt{2}$, which boils
down to $y>n(1-x)$ where $n:=\frac{m}{2\sqrt{2}+1}$. So, in the
above inequality involving $x$, the range of $t$ is $[n,m/3]$. One
can check that $f$ is well defined on this interval. In addition,
$f(n)=1$ and $f(m/3)=0$. The other inequality in (\ref{abm})
reduces to $x>g(t)$ with $g$ defined as in the statement. We
denote by $\mathcal R$ the region we are interested in, i.e. the
right-hand petal going down in Figure~\ref{Figure8}. The Jacobian
of the transformation $(x,y)\to (x,t)$ is $J=1-x$ and so

$$Area(\mathcal R)=\underset{\mathcal R}{\int\int}dxdy=\int_n^{m/3}\int_{g(t)}^{f(t)}(1-x)dxdt=\frac{1}{2}\int_n^{m/3}(1-g(t))^2-(1-f(t))^2dt$$

\n which implies

$$P=Area(\mathcal R)/(\frac{m}{6})=m\int_n^{m/3}(1-g(t))^2-(1-f(t))^2dt.$$

\subsection{Distances to the sides from the circumcenter.}\label{doa}

In this section there will be no need to compute any probabilities.
Part of the next theorem appeared as a proposed problem in \cite{eji1}.

\begin{theorem}\label{dtosidesdfromcirc}  Consider a  triangle $ABC$ and let
$O$ be its  circumcenter. Denote the distances of $O$ to the sides
$\overline{BC}$,  $\overline{AC}$, and $\overline{AB}$, by $u$,
$v$ and $w$ respectively.

{\rm (i)} The radius $R$, of the circle circumscribed to the
triangle $ABC$, satisfies the equation
\begin{equation}\label{maineq}
R^3-(u^2+v^2+w^2)R-2uvw=0,\ if\ \triangle ABC\ is \ acute;
\end{equation}

\begin{equation}\label{maineq2}
R^3-(u^2+v^2+w^2)R+2uvw=0,\ if\ \triangle ABC\ is \ obtuse;
\end{equation}

\n  and (obviously)

\begin{equation}\label{maineq3}
R=(u^2+v^2+w^2)^{\frac{1}{2}},\ if\ \triangle ABC\ is\ a\ right \
triangle.
\end{equation}

{\rm (ii)} Given three positive  real numbers $u$, $v$, and $w$,
there exists one and only one acute triangle with the distances of
the circumcenter to the sides equal to  $u$, $v$ and $w$. The
previous statement is true if one changes the adjective acute to
obtuse.

{\rm (iii)} The equation {\rm (\ref{maineq})} has infinitely many
integer solutions $(u,v,w,R)\in \mathbb N^4$ such that $u$, $v$,
and $w$ are all different.
\end{theorem}

\vspace{0.1in}

\n \begin{proof} (i) Denote by $D$, $E$ and $F$ the projections of
$O$ on $\overline{AC}$, $\overline{AB}$, and $\overline{BC}$
respectively (see Figure~\ref{figure6} (a)). One can easily prove
 the identity

\begin{equation} \label{identity}
\cos A+\cos B+\cos C=1+4\sin \frac{A}{2} \sin \frac{B}{2} \sin
\frac{C}{2},
\end{equation}

\n where $A$, $B$ an $C$ are the angles of the triangle. In the
triangle $\triangle OBC$,   $\overline {OF}$ is clearly the angle
bisector of $\angle BOC$. First, we assume that the triangle $ABC$
is acute. Because $A$ is less than $90^{\circ}$, $m(\angle
BOF)=\frac{m (arc\ \overset{\frown} {BC})}{2}=A$. Hence $\cos
A=\frac{u}{R}$,  and similarly $\cos B=\frac{v}{R}$, and  $\cos
C=\frac{w}{R}$. Substituting into (\ref{identity}) we get

\begin{equation}\label{secondeq}
u+v+w=R+\sqrt{\frac{2(R-u)(R-v)(R-w)}{R}},
\end{equation}

\n which after elimination of the radical sign gives the equation
(\ref{maineq}).

(ii) First we want to show the existence and uniqueness of an
acute triangle with the required property. Let us denote  the
quantity $\left(\frac{u^2+v^2+w^2}{3}\right)^{1/2}$ by $\omega$
and observe that the AM-GM Inequality gives
$$\omega^3 \ge uvw.$$

\n If we consider the cubic polynomial function
$$f(t)=t^3-(u^2+v^2+w^2)t-2uvw,\ t\in \mathbb R$$
\n observe that $f'$ has as critical points $\pm \omega$. There
are clearly at most three real solutions of $f(t)=0$. Since
$f(0)=-2uvw<0$ and $f(-\omega)=2(\omega^3-uvw) \ge 0$, $f$ must
have two real zeros in $(-\infty, 0)$ (or possibly one with
multiplicity two) and a unique positive zero that we will simply
denote by $R$. Because $f(u)=-u(v+w)^2<0$ , $f(v)=-v(u+w)^2<0$,
$f(w)=-w(u+v)^2<0$ and $f(2\omega)=2(\omega^3-uvw)\ge 0$ we see
that
$$R\in (\max \{x,y,z\},2\omega].$$

\begin{figure}
\begin{center}
$\underset{(a)\ \text{Distances to sides from the circumcenter}
}{\epsfig{file=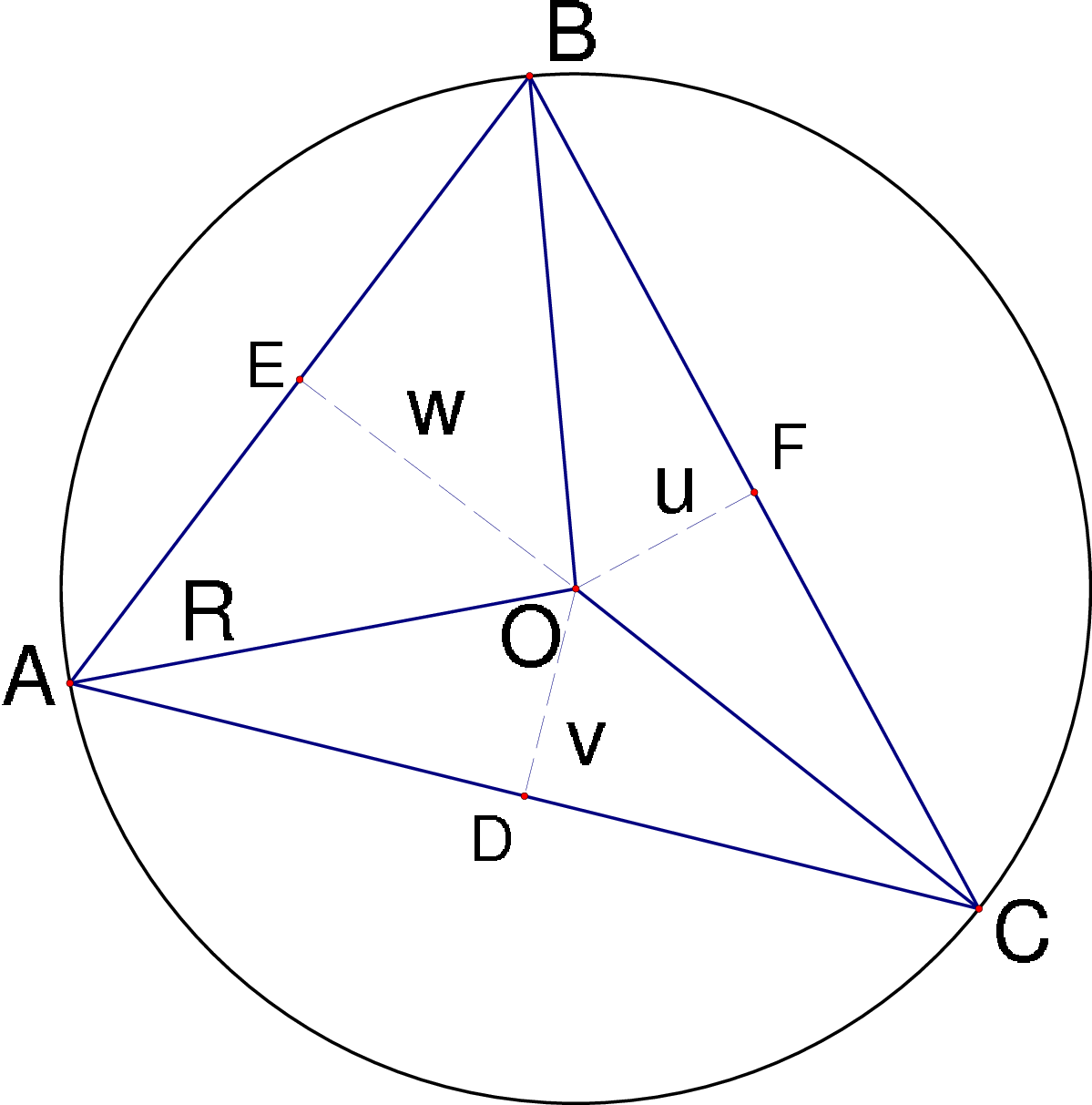,height=2.2in,width=2.2in}}$\ \ \ $\underset{(b)\
\text{Orthocenter's distances to vertices}
}{\epsfig{file=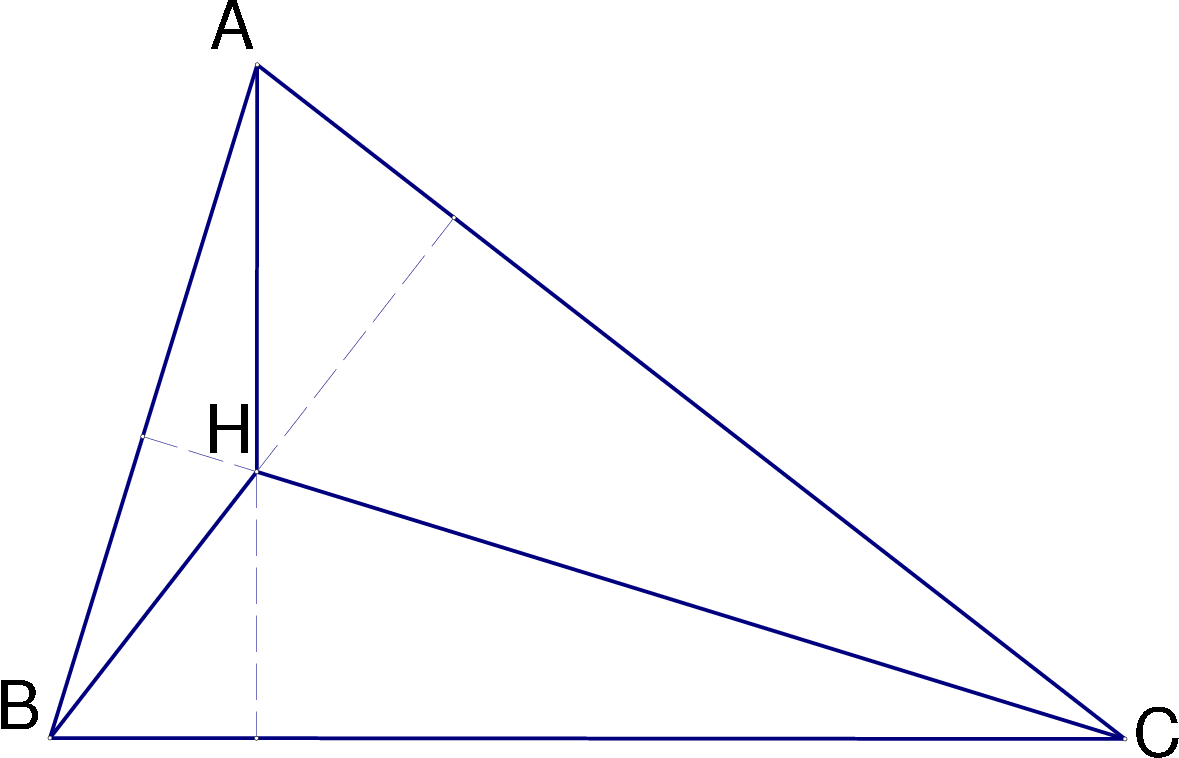,height=1.9in,width=2.7in}}$
\end{center}
\caption{Identical problems} \label{figure6}
\end{figure}

The radius $R$ determines the sides $a$, $b$ and $c$ by the
formulas $a=2R\sin A=2\sqrt{R^2-u^2}$, $b=2\sqrt{R^2-v^2}$ and
$c=2\sqrt{R^2-w^2}$. Without loss of generality we may assume that
$w\le v\le u$. In order to have a triangle with side lengths $a$,
$b$ and $c$ it is necessary and sufficient to  have

$$\sqrt{R^2-u^2}+\sqrt{R^2-v^2}>\sqrt{R^2-w^2}$$

or

$$2\sqrt{(R^2-u^2)(R^2-v^2)}>u^2+v^2-w^2-R^2.$$

This is trivially verified if we show that
$R>(u^2+v^2-w^2)^{1/2}$. Since
$f((u^2+v^2-w^2)^{1/2})=-2w^2(u^2+v^2-w^2)^{1/2}-2uvw<0$ then we
must have $R>(u^2+v^2-w^2)^{1/2}$. Once we have the triangle
constructed with side lengths $a$, $b$ and $c$, we must check to
see if the triangle is acute, i.e. $a^2+b^2>c^2$, $b^2+c^2>a^2$
and $a^2+c^2>b^2$. These inequalities are equivalent to
$R>(u^2+v^2-w^2)^{1/2}$, $R>(u^2-v^2+w^2)^{1/2}$ and
$R>(-u^2+v^2+w^2)^{1/2}$ respectively, which were shown to be true
earlier. We denote the angles of the triangle with sides $a$, $b$,
and $c$ by $A'$, $B'$ and $C'$. If we calculate the cosine
function for $A'$ we get

$$\cos
A'=\frac{b^2+c^2-a^2}{2bc}=\frac{R^2+u^2-v^2-w^2}{2\sqrt{(R^2-v^2)(R^2-w^2)}}.$$

Using (\ref{maineq}), which $R$ satisfies, one can show that $\cos
A'=\frac{u}{R}$. So, $\sin A'=\frac{a}{2R}$ which implies that $R$
is the radius of the circle circumscribed about the constructed
triangle. Then the distances to the sides from the center of the
circumscribed circle  must be $u$, $v$ and $w$. Therefore, we have
only one triangle that satisfies the required conditions.

For the second part of the claim in (ii) one needs to repeat the
above arguments with the appropriate changes. In this case the
radius $R$ must satisfy the equation

\begin{equation}\label{maineq2}
R^3-(u^2+v^2+w^2)R+2uvw=0.
\end{equation}

(iii) One such solution is $u=2$, $v=7$, $w=11$ and $R=14$. This
example suggests that one can take $R=uv$ and hope to obtain more
solutions of this type. In this case, (\ref{maineq}) reduces to

$$(u^2-1)(v^2-1)=(w+1)^2.$$

This equation is satisfied if $v^2-1=k^2(u^2-1)$ for some $k\in
\mathbb N$, and $zw=k(u^2-1)-1$. If we fix $u=2$ for instance, we
get Pell's diophantine equation, $v^2-3k^2=1$, which  is known to
have infinitely many integer solutions. The values $u$, $v$ and
$w$ are clearly distinct if $k>1$.

There are many different patterns of solutions. Some examples are
included in the table below:

\vspace{0.1in}
\begin{center}
\begin{tabular}{|c||c|c|c|c|c|c|c|c|c|c|c|c|}
  \hline
  u& 1 & 2 & 2 & 3 & 4 & 4 & 6 & 7 & 8 & 11 & 11 & 12 \\   \hline
  v & 13 & 7 & 9 & 14 & 14 & 18 & 11 & 19 & 17 & 17 & 19 & 22 \\   \hline
  w & 22 & 11 & 12 & 25 & 22 & 24 & 14 & 25 & 22 & 21 & 26 & 28 \\  \hline
  R & 26 & 14 & 16 & 30 & 28 & 32 & 21 & 35 & 32 & 33 & 38 & 42 \\
  \hline
\end{tabular}
\end{center}
\end{proof}
Let us observe that this discussion of this subsection also solves
the problem for $\alpha=HA$, $\beta=HB$ and $\gamma=HC$, where $H$
is the orthocenter of a triangle, i.e. the intersection of its
altitudes (see Figure~\ref{figure6}(b)). Indeed, one can show
that there are very similar formulas for these distances in terms
of the sides and angles of the triangle: $HA=2R\cos A$, $HB=2R\cos
B$ and $HC=2R\cos C$. Similarly, the problem $\alpha=HD$, $\beta=HE$ and $\gamma=HF$
(Figure~\ref{Figure7}(a)) leads to the same analysis since $HD=2R\cos A\cos C$.

\begin{figure}
\begin{center}
$\underset{(a)\ \text{Distances to sides from the circumcenter}
}{\epsfig{file=fig1v2.eps,height=2.2in,width=2.2in}}$$\underset{(b)\
\text{The incircle}
}{\epsfig{file=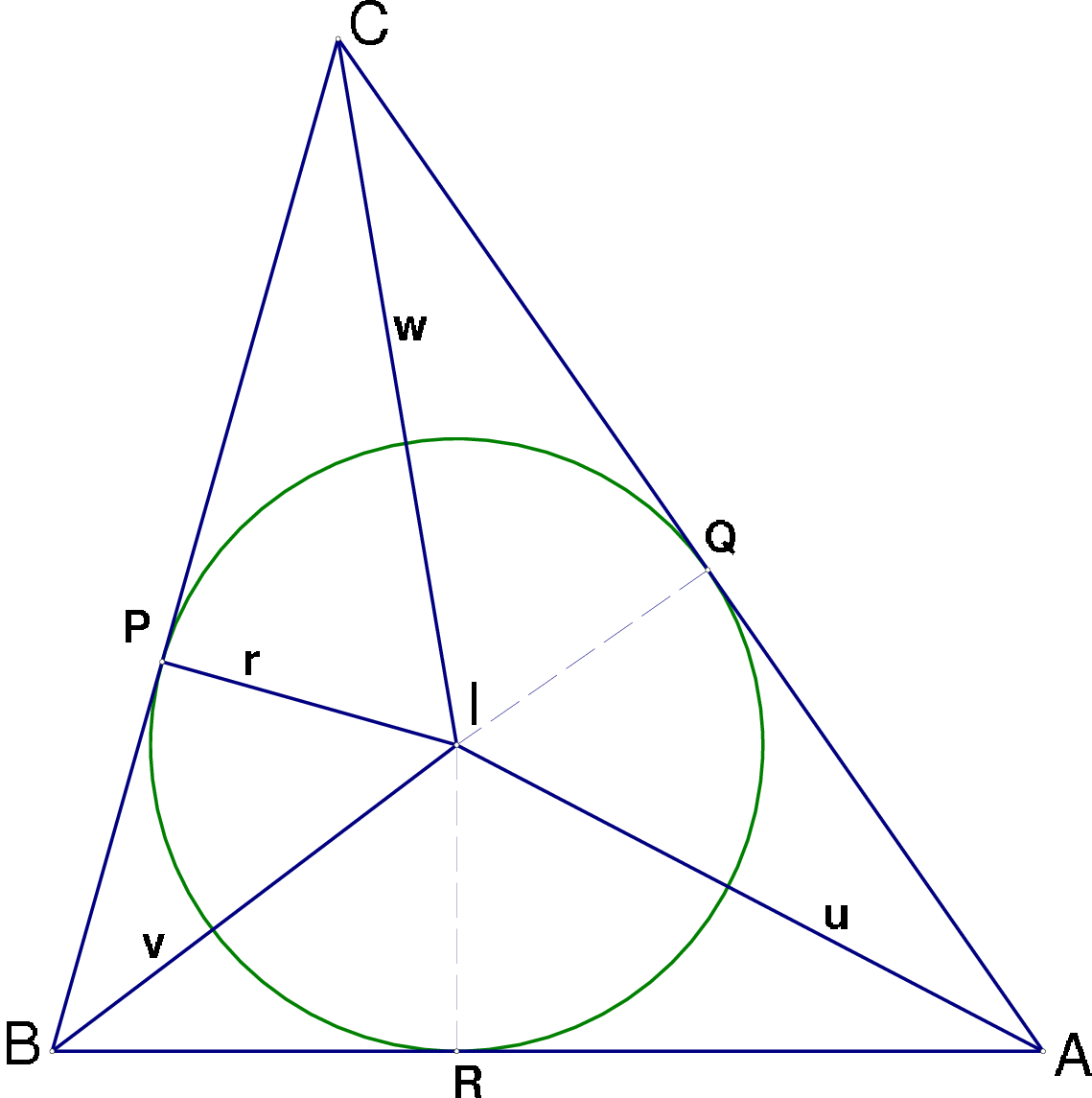,height=2.3in,width=2.3in}}$

\end{center}
\caption{ Special Cases} \label{figure6}
\end{figure}

\subsection{Distances from the center of the incircle to the
vertices}

First we will show a relation between the radius of the incircle
and the distances from the center of the incircle to the vertices.
If $I$ is the center of the incircle of the triangle $ABC$, $r$ is the
radius of the incircle and we denote $AI$, $BI$ and $CI$ by $u$, $v$ and
$w$ respectively as in Figure~\ref{figure6}(b), then
\[
\sin \frac{A}{2} = \frac{r}{u}, \quad \sin \frac{B}{2} =
\frac{r}{v}, \quad \textup{and} \quad \sin \frac{C}{2} =
\frac{r}{w}.
\]

From here we infer that
\[
\cos \frac{B}{2} = \sqrt{1 - \frac{r^2}{v^2}} \quad \textup{and}
\quad \cos \frac{C}{2} = \sqrt{1 - \frac{r^2}{w^2}}.
\]
On the other hand, since
\[
\sin \frac{A}{2} = \sin \frac{\pi - (B + C)}{2} = \cos \frac{B +
C}{2} = \cos \frac{B}{2} \cos \frac{C}{2} - \sin \frac{B}{2} \sin
\frac{C}{2}
\]
we get the third degree equation in $r$ (with $u$, $v$ and $w$ as
parameters)
\[
\frac{r}{u} = \sqrt{1 - \frac{r^2}{v^2}}\sqrt{1 - \frac{r^2}{w^2}}
- \frac{r^2}{vw} \iff \frac{r}{u} + \frac{r^2}{vw} = \sqrt{1 -
\frac{r^2}{v^2}}\sqrt{1 - \frac{r^2}{w^2}}
\]
\[
\iff \frac{r^2}{u^2} + \frac{r^4}{v^2 w^2} + 2 \frac{r^3}{uvw} = 1
- \frac{r^2}{v^2} - \frac{r^2}{w^2} + \frac{r^4}{v^2 w^2} \iff
\]

\begin{equation}\label{cubicinr}
\frac{2}{uvw} r^3 + \left(\frac{1}{u^2} + \frac{1}{v^2} +
\frac{1}{w^2} \right) r^2 - 1 = 0.
\end{equation}

\n It is easy to see that equation (\ref{cubicinr}) has a unique
positive solution which is  less than either of the values $u$,
$v$, or $w$.

Once we have $r$, a simple geometrical construction shows that
$a$, $b$ and $c$ are uniquely determined by $u$, $v$ and $w$. We
want to show some relations between the sides of the triangle, the
radius of the incircle, and the distances from the center of the
incircle to the vertices that will make this clear. Let $P, Q, R$
be the points of intersection of the perpendiculars from $I$ on
$BC$, $CA$ and $AB$ respectively. It is well known that
\[
PC = QC = \frac{a + b - c}{2}; \ QA = RA = \frac{b + c - a}{2}, \
\text{and} \ RB = PB = \frac{a + c - b}{2}.
\]
Then
\[
a + b - c = 2 \sqrt{w^2 - r^2}, \ b + c - a = 2 \sqrt{u^2 - r^2},
\ \textup{and} \ a + c - b = 2 \sqrt{v^2 - r^2},
\]
which leads to
\begin{equation}\label{abcintermsofuvwandr}
\begin{cases} a = \sqrt{v^2 - r^2} + \sqrt{w^2 - r^2}, \\ \\  b = \sqrt{u^2
- r^2} + \sqrt{w^2 - r^2}, \ \text{and}\  \\ \\ c = \sqrt{u^2 -
r^2} + \sqrt{v^2 - r^2}.
\end{cases}
\end{equation}

Now we will work our way backwards.

\begin{theorem}\label{distancesfromI}
If $u$, $v$, and $w$ are positive quantities then there is a
unique triangle such that the distances from the vertices to the
center of the incircle are equal to $u$, $v$, and $w$
respectively.
\end{theorem}
\begin{proof} The part about uniqueness follows from the analysis above the statement of the theorem.
For existence, we let $r$ be the unique positive solution of
(\ref{cubicinr}) and $a$, $b$ and $c$ as given by
(\ref{abcintermsofuvwandr}). Then, using Heron's formula, the area
of the triangle is given by
\[
A = \sqrt{\sqrt{u^2 - r^2} \sqrt{v^2 - r^2} \sqrt{w^2 - r^2} (
\sqrt{u^2 - r^2} + \sqrt{v^2 - r^2} + \sqrt{w^2 - r^2})}
\]
and hence the radius of the circle inscribed in the triangle with
sides $a$, $b$ and $c$ is
\[
r' =\frac{2A}{a+b+c}= \sqrt{\frac{\sqrt{u^2 - r^2} \sqrt{v^2 -
r^2} \sqrt{w^2 - r^2}}{\sqrt{u^2 - r^2} + \sqrt{v^2 - r^2} +
\sqrt{w^2 - r^2}}}.
\]

\n With the analysis we did earlier, we see that if $r'=r$, the
Pythagorean theorem and formulas (\ref{abcintermsofuvwandr}) will
give  $AI=u$, $BI=v$ and $CI=w$. So, to complete the proof we need
to show that $r = r'$. In other words, we must show is that
(\ref{cubicinr}) implies
\[
r = \sqrt{\frac{\sqrt{u^2 - r^2} \sqrt{v^2 - r^2} \sqrt{w^2 -
r^2}}{\sqrt{u^2 - r^2} + \sqrt{v^2 - r^2} + \sqrt{w^2 - r^2}}}.
\]

\n Because $r$ is less than each of $u$, $v$ and $w$, the
substitutions
\[
m = \frac{u}{r}>1, \quad n = \frac{v}{r}>1, \quad  \text{and}\ \ p
= \frac{w}{r}>1,
\]
make the last equality equivalent to

\begin{equation}\label{eq1}
\sqrt{m^2 - 1} + \sqrt{n^2 - 1} + \sqrt{p^2 - 1} = \sqrt{m^2 - 1}
\sqrt{n^2 - 1} \sqrt{p^2 - 1}.
\end{equation}
We note that, with these substitutions, (\ref{cubicinr}) becomes

\begin{equation}\label{eq2}
2mnp + m^2 n^2 + n^2 p^2 + m^2 p^2 - m^2 n^2 p^2 = 0.
\end{equation}

\n Eliminating the square roots in a careful way, (\ref{eq1})
becomes

\[
\sqrt{m^2 - 1} + \sqrt{n^2 - 1}  = (\sqrt{m^2 - 1} \sqrt{n^2 -
1}-1)\sqrt{p^2 - 1}
\]
\[
\iff m^2 - 1 + n^2 - 1 + 2 \sqrt{m^2 - 1} \sqrt{n^2 - 1} = (m^2 -
1) (n^2 - 1) (p^2 - 1) + p^2 - 1
\]
\[
- 2 (p^2 - 1) \sqrt{m^2 - 1} \sqrt{n^2 - 1}.
\]

\n Using (\ref{eq2}) this last equality simplifies to $p \sqrt{m^2
- 1} \sqrt{n^2 - 1} = m n + p$ which, after getting rid of the
square roots, becomes (\ref{eq2}). \end{proof}

\begin{corollary}\label{ascutitunghic} Given $u$, $v$ and $w$
three positive real numbers, the triangle ensured by
Theorem~\ref{distancesfromI} is acute if and only if

\begin{equation}\label{ascutitunghicI}
\begin{cases}
\sqrt{2}u^2vw + u^2(v^2 + w^2)- v^2w^2 > 0,\\ \\
\sqrt{2}u v^2 w + v^2(u^2 + w^2)- u^2w^2 > 0, \ and  \\ \\
\sqrt{2} uv w^2 + w^2(u^2 + v^2)- u^2v^2 > 0.
\end{cases}
\end{equation}
In the context of the broken stick problem, if $u=\alpha$,
$v=\beta$ and $w=\gamma$, the probability that the triangle given
by Theorem~\ref{distancesfromI} is acute is approximately
$0.1962$.
\end{corollary}

\proof  The triangle is acute if and only if
\[
\sin \frac{A}{2} < \frac{1}{\sqrt{2}}, \ \sin \frac{B}{2} <
\frac{1}{\sqrt{2}}, \quad \textup{and} \quad \sin \frac{C}{2} <
\frac{1}{\sqrt{2}}.
\]
This means that
\[
\frac{r}{u} < \frac{1}{\sqrt{2}}, \ \frac{r}{v} <
\frac{1}{\sqrt{2}},\quad \textup{and} \quad \frac{r}{w} <
\frac{1}{\sqrt{2}} \iff r < \min \left\{ \frac{u}{\sqrt{2}},
\frac{v}{\sqrt{2}}, \frac{w}{\sqrt{2}} \right\}.
\]

Since $r$ is the unique positive solution of (\ref{cubicinr}) and
the derivative of the function
$g(t)=\frac{2t^3}{uvw}+(1/u^2+1/v^2+1/w^2)t^2-1$ is positive for
$t>0$,  this is equivalent to $g(\frac{u}{\sqrt{2}})>0$,
$g(\frac{v}{\sqrt{2}})>0$, and $g(\frac{w}{\sqrt{2}})>0$. This
translates into (\ref{ascutitunghicI}). The equations that define
the probability are of degree four and we could only find the
probability experimentally. \eproof

The ratio $P(obtuse)/P(acute) \approx 4.1$ balances out the previous cases.

\section{Further problems and a summary}\label{challengeproblems}
One may investigate using this technique the case in which $\alpha$, $\beta$ and $\gamma$ are the symmedians
of a triangle. The formula for the symmedian corresponding to vertex $A$ is given by $sm_a=\frac{bc}{b^2+c^2}\sqrt{2(b^2+c^2)-a^2}$.
This formula is very similar to the angle bisector formula but the situation seems
to be very different of the one discussed in Section~\ref{AngleBisectors}. We have no answer to this problem.

\n There are certainly interesting generalizations that
can be considered and in some directions they have already appeared in the literature. For instance, Carlos D'Andrea and Emiliano Gomez
(\cite{dg}) showed that if $n-1$ ($n\ge 3$) breaking points are
considered, the probability of having an $n$-gon with the resulting
segments is equal to $\ds 1-n/2^{n-1}$. This result also
appeared in \cite{bb}, where the solution is derived by solving
another geometric probability question, called by the authors, {\em
The Semicircle Problem} (\cite{mgb}). In fact, it was shown to be
 equivalent to this problem: {\it ``If $n + 1$ points are randomly
selected on the circumference of a circle, what is the probability
that they will all fall within some semicircle?"} Another direction
of further investigations along these lines is to go into space, so to speak, and ask: ``{\it If the stick breaks into six segments, what is the probability
that the segments are the sides of a tetrahedron?}"

Let us briefly discuss the following generalization which appeared in \cite{vv} as a proposed problem. It is worth mentioning that the solution to this problem was
from its author, Professor Gheorghe Mihoc, and it was based on a
different idea than the one we have included below.

\begin{proposition}\label{gheMihoc} Given an arbitrary triangle
with sides $a$, $b$ and $c$, the probability that the distances
from a point inside the triangle to the sides of the triangle form
a triangle, is equal to
$$\frac{2abc}{(a+b)(b+c)(c+a)}.$$
\end{proposition}

\begin{figure}[h]
$$\underset{ Figure\ 9:\  \text{$DEF$ triangle}
}{\epsfig{file=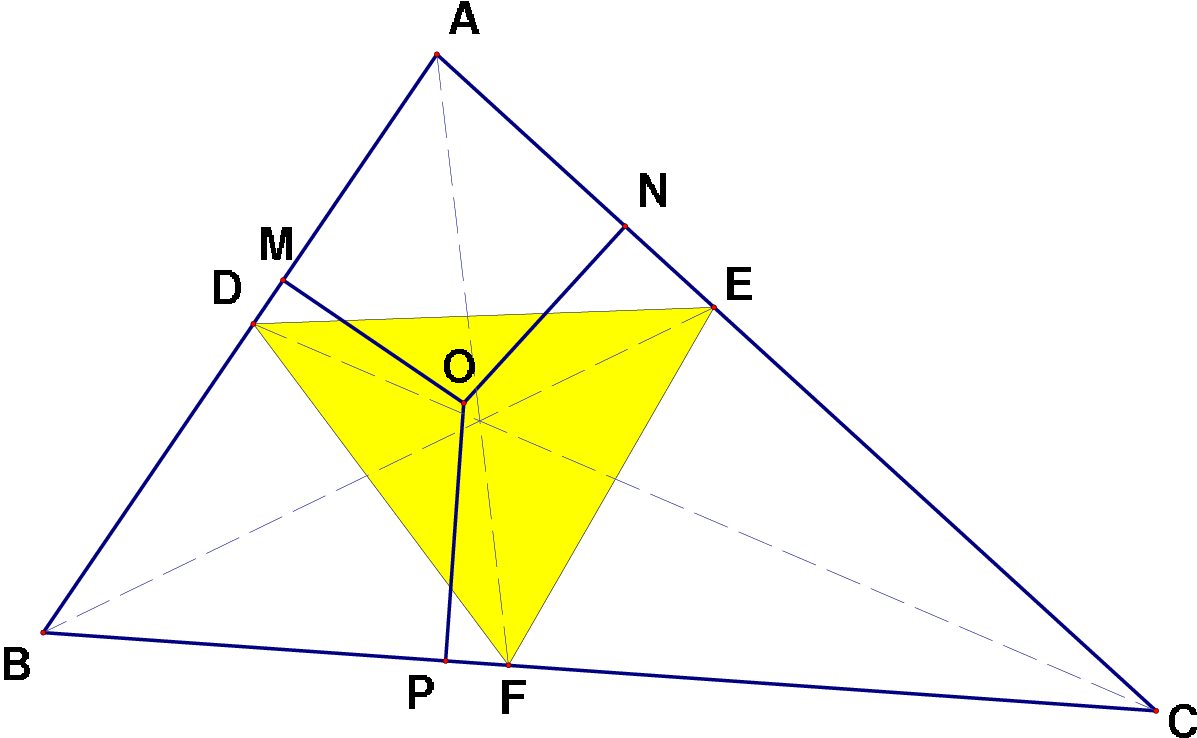,height=1.6in,width=2.2in}}$$\
\label{Figurelast}
\end{figure}

For the sake of completeness we sketch a proof of
this proposition. We  refer to Figure~9. Triangle $DEF$ is given by the points where the angle
bisectors intersect the sides of triangle $ABC$. First, one
shows that the region determined by the interior of triangle $DEF$ is the region that gives the desired probability. Using the
Angle Bisector Theorem one can show that

$$\frac{Area(\triangle BDF)}{Area(\triangle ABC)}=\frac{ac}{(b+c)(a+b)},$$

\n and  also the other two equalities obtained by cyclic
permutation of the sides $a$, $b$ and $c$. The formula given, now
follows from an algebraic identity

$$(a+b)(b+c)(c+a)-[ (a+b)ab+(b+c)bc+(c+a)ca]=2abc.$$

\n We observe that this probability has its greatest value of
$1/4$ when $a=b=c$. This means, the probability is at a maximum
when $ABC$ is equilateral.

Another question one may ask is: ``how does the answer in
Theorem~\ref{sidesacute} change if $\alpha$, $\beta$ and $\gamma$
are computed relative to an arbitrary triangle as in
Proposition~\ref{gheMihoc}?" A general answer is probably quite complicated because the curve
$\alpha^2=\beta^2+\gamma^2$ may be an arc of an ellipse or an arc of a hyperbola. For example, if
$a=b=\frac{15}{4}$ and $c=6$, one of the conics at the boundary of
the region defining the probability is an ellipse and the other
two are hyperbolas. The probability in Theorem~\ref{sidesacute}
becomes

$$P=\frac{25}{28}+\frac{25}{32}\ln
\frac{13}{5}-\frac{100}{49} \sqrt{14}
\arcsin\left(\frac{\sqrt{7}}{13}\right).$$

\vspace{0.2in}
Finally, let us summarize  our results:

\vspace{0.2in}
\centerline{
\begin{tabular}{||c|c|c|c||}
 \hline \hline
{\bf Case}  & {\bf Probability}  & {\bf Acute}  & {\bf Ratio $\frac{Obtuse}{Acute}$}\\
  \hline \hline
  classical case & $\frac{1}{4} $& $\ln(\frac{8}{e^2})$ & 2.146968 \\
  \hline
  medians &  $\frac{1}{4} $ & $\frac{1}{3}-\frac{5}{9}\ln\left(\frac{8}{5}\right)$  &  2.461635121 \\
  \hline
  altitudes & $\frac{4}{25}\left(3\sqrt{5}\ln \frac{3+\sqrt{5}}{2}-5\right)$&  0.07744388 & 2.008 \\
  \hline
r,s,t  & $\frac{5}{27}$ & 0.047 &  2.87  \\
\hline
angle bisector & 1 & 0.1195 & 7.36\\
\hline
$IA$, $IB$, $IC$ & 1 &  0.1962 &  4.1 \\
\hline
excircles radii& 1 & $\frac{24\sqrt{7}}{49}\arcsin(\frac{\sqrt{14}}{8})-\frac{2}{7} $ &  1.9  \\
\hline
$h_a$, $w_a$ and $m_a$ & 1 &  0.042234  & 22.7 \\
\hline
\hline
\end{tabular}}
\vspace{0.2in}
\n {\bf Acknowledgements.} We thank Professor
Albert~VanCleave who helped us significantly improve the
exposition of this article. We also want to thank Professors Charles M. Grinstead and J. Laurie Snell for writing the book
\cite{gs} and providing that in pdf format. The first author learned about The Broken Stick Problem from their book (Problem 12, page 73)
while teaching the Introduction to Probability course at CSU in the Spring of 2008.

\n Eugen J. Ionascu \texttt{ Department of Mathematics, Columbus State University, Columbus, GA
31907, US},  e-mail: {math@ejionascu.ro}

\vspace{0.1in}

\n Gabriel Prajitura \texttt{  The College at Brockport, State University of New York, US}
e-mail: {gprajitu@brockport.edu}

\end{document}